\documentclass[12pt,a4paper,oneside]{article}
\usepackage{amsmath,amstext,amssymb,amsopn,amsthm,color}
\usepackage{graphicx}
\usepackage{hyperref}

\usepackage{tikz}
\usetikzlibrary{calc}
\usetikzlibrary{arrows}
\usetikzlibrary{patterns}
\usetikzlibrary{through}
\usetikzlibrary{plotmarks}
\usetikzlibrary{shapes}
\usetikzlibrary{svg.path}

\usepackage{ushort}

\definecolor{kb}{rgb}{0.1,0.5,0.1}
\definecolor{mr}{rgb}{0.1,0.2,0.7}
\definecolor{tg}{rgb}{0.7,0.1,0.2}

\newcommand{\gener}{{\cal A}}
\newcommand{\indyk}[1]{{\bf 1}_{#1}}

\theoremstyle{plain}

 \newtheorem{thm}{Theorem}[section] 
\newtheorem{lem}[thm]{Lemma} \newtheorem{prop}[thm]{Proposition}
\newtheorem{cor}[thm]{Corollary}
 \theoremstyle{definition}
\newtheorem{definition}{Definition}

\newtheorem{exmp}{Example} \theoremstyle{remark}
\newtheorem*{rem*}{Remark} \newtheorem{remark}{Remark}

\numberwithin{equation}{section}

 \newcommand{\R}{\mathbb{R}}

\newcommand{\RR}{\R} \newcommand{\Rd}{{\R^{d}}}
\newcommand{\Rdz}{{\R^{d}\setminus\{0\}}}
\newcommand{\N}{\mathbb{N}}

\renewcommand{\H}{\mathbb{H}} 
\renewcommand{\leq}{\leqslant} \renewcommand{\le}{\leq}
\renewcommand{\geq}{\geqslant} \renewcommand{\ge}{\geq}

\newcommand{\Z}{\int^{\infty}_{0}}

\DeclareMathOperator{\CAP}{{\it{Cap}}}
\DeclareMathOperator{\dist}{dist}
\DeclareMathOperator{\diam}{diam}

\DeclareMathOperator{\esssup}{ess\,sup}
   
 \def\({\left(} \def\){\right)} \def\[{\left[}
  \def\]{\right]} \def\<{\langle} \def\>{\rangle}

\newcommand{\E}{\mathbb{E}}
\newcommand{\p}{\mathbb{P}}

\ushortCreate*(3)(.46){uline}

\newcommand{\A}{({\bf H})}
\newcommand{\As}{({\bf H^*})}

\newcommand{\WUSC}[3]{\textrm{WUSC}(#1,#2,#3)}
\newcommand{\WLSC}[3]{\textrm{WLSC}(#1,#2,#3)}
\newcommand{\lC}{{\underline{c}}}
\newcommand{\uC}{{\overline{C}}}
\newcommand{\la}{{\underline{\alpha}}}
\newcommand{\ua}{{\overline{\alpha}}}
\newcommand{\lt}{{\underline{\theta}}}
\newcommand{\ut}{{\overline{\theta}}}

\title{Barriers,
exit time and survival probability for unimodal L\'evy processes
}

\author{K. Bogdan\thanks{corresponding author, Institute of Mathematics of Polish Academy of Sciences and Institute of Mathematics and Computer Science,
  Wroc\l{}aw University of Technology,  ul.~Wyb. Wyspia\'{n}skiego
27, 50-370 Wroc\l{}aw, Poland, krzysztof.bogdan@pwr.wroc.pl, tel. +48 71 320 3180}, T. Grzywny\thanks{Institute of Mathematics and Computer Science,
  Wroc\l{}aw University of Technology, ul. Wyb. Wyspia\'{n}skiego
27, 50-370 Wroc\l{}aw, Poland, tomasz.{grzywny}@pwr,wroc.pl, michal.ryznar@pwr.wroc.pl}, M. Ryznar\textsuperscript{\textdaggerdbl}
 }

\begin{document}
\maketitle
\begin{abstract}
We construct superharmonic functions and give sharp bounds for the expected exit time and probability of survival
for isotropic unimodal L\'evy processes
in smooth domains.
\end{abstract}

\section{Introduction}\label{sec:i}

A function is called 
barrier for an open set if it is superharmonic inside and vanishes outside, near a part of the boundary of the set. Barriers are important for studying boundary 
behavior of solutions to the Dirichlet problem \cite{MR1801253,MR856511}.
From a general perspective, understanding boundary 
asymptotics of superharmonic functions 
gives detailed information on the behavior of the underlying Markov process at the boundary. The information is obtained by using maximum principle, super-mean value  property and Doob's conditioning.
Calculation of barriers is extremely delicate for open sets with Lipschitz regularity, even for the Laplacian and cones in $\Rd$, see, e.g., \cite[Section~3]{MR0474525}, \cite{MR2904138}. The situation is somewhat easier for smooth open sets. For instance, the Laplacian in a half-space has  barriers which are linear functions, correspondingly for smooth sets  {approximately linear} barriers exist. Similar results, with non-linear boundary decay, are known for the fractional Laplacian and generators of convolution semigroups corresponding to {\it complete subordinate Brownian motions} with {{\it weak scaling} (see \cite{MR2213639, MR2928332} and Section~\ref{sec:conditionA} for discussion and references).
Recall that for a sub-Markovian semigroup $(P_t,t\ge 0)$ we have $\gener f(x)=\lim_{t\to 0^+} [P_tf(x)-f(x)]/t\le 0$ if $f$ is bounded, the limit exists and $f(x)=\max f\ge 0$.
Accordingly, we say that operator $\gener$ on $C^\infty_c(\Rd)$
satisfies the {\it positive maximum principle} if
for every $\varphi \in C^\infty_c(\Rd)$, $\varphi(x)=\sup_{y\in \Rd} \varphi(y)\geq 0$ implies  $\gener \varphi(x)\leq 0$.
The most general operators 
which have this
property are of the form
\begin{eqnarray*}
\gener\varphi(x)&=&\sum_{i,j=1}^d a_{ij}(x)D_{x_i}D_{x_j} \varphi(x)
+b(x) \nabla \varphi(x) +q(x)\varphi(x)\nonumber\\
&&+
\int\limits_{\Rd}\left(\varphi(x+y)-\varphi(x)-
y\nabla \varphi(x)\;\indyk{|y|<1}
\right)\,
  \nu(x,dy)\,.\label{eq:courr}
\end{eqnarray*}
Here 
for every $x\in \Rd$, $a(x)=(a_{ij}(x))_{i,j=1}^n$ is a real nonnegative definite symmetric matrix,
vector $b(x)=(b_i(x))_{i=1}^d$ has real coordinates, $q(x)\leq 0$,
and $\nu(x,\cdot)$ is a L\'evy measure.\index{L\'evy measure}
The description is due to Courr\`ege, see, e.g., \cite[Proposition 2.10]{Hoh98}.
For translation invariant (convolution) operators of this type, $a$, $b$, $q$,
and $\nu$ are independent of $x$. If we further assume rotation invariance
and conservativeness ($\gener 1=0$), then
\begin{eqnarray}
\gener \varphi(x)&=&\sigma \Delta \varphi(x)
+\lim_{\varepsilon \to 0^+} \int\limits_{|y|>\varepsilon}\left[\varphi(x+y)-\varphi(x)\right]
  \nu(dy)\,,\label{eq:courr1}
\end{eqnarray}
 where $\sigma\geq0$ and $\nu$ is {\it isotropic}. \eqref{eq:courr1} gives
the general setting of our paper; we shall also consider the corresponding {isotropic} L\'evy processes $X$.

It is in general difficult to determine barriers for non-local Markov generators, even in the setting of \eqref{eq:courr1} and for smooth open sets. In fact the wide range of L\'evy measures $\nu$ results in a comparable variety of boundary asymptotics of superharmonic functions, not fully codified by the existing calculus. The situation might even seem hopeless but it is not.
For instance, the expected exit time $x\mapsto\E_x \tau_D$ of $X$ from open bounded set $D\subset \Rd$
is a barrier for $D$. The function
we shall effectively estimate this function for smooth open sets $D$ and {\it unimodal} L\'evy processes $X$ by constructing barriers for the ball
of arbitrary radius. 
To this end we use the renewal function $V$ of the ladder-height process of one-dimensional projections of $X$: the barriers are defined as compositions of $V$ with the distance to the complement of the ball.
This and a similar construction of functions subharmonic in the complement of the ball
yield sharp estimates for the expected exit time  
for open sets
$D\subset \Rd$ which are of class $C^{1,1}$. 
We also obtain
sharp estimates for the probability of $X$ surviving in $D$ longer than given time $t>0$, even for some unbounded $D$ and rather general {\it unimodal} L\'evy processes.

Thus, $V$ allows for calculations accurate enough  to exhibit specific super- and subharmonic functions for the considered processes. 
The idea of using $V$
in this context comes from
P.~Kim, R.~Song and Z.~Vondra\^{c}ek \cite{MR2928332} (see Introduction and p.~931 ibid.) and has already proved
very fruitful for {\it complete subordinate Brownian motions}.

When verifying superharmonicity, we calculate a version of the infinitesimal generator on
the composition of $V$ with the distance to the complement of the ball. In view if the curvature of the sphere, the calculation requires good control of $V'$.
We carry out calculations assuming that
$V'$ satisfies 
a Harnack-type
condition 
$\A$, described in \eqref{HR} below. 
When using
$\A$
we only need to estimate certain weighted integrals of $V'$ (given, e.g., by Lemma~\ref{VIntegralEstimate}), rather than individual values of $V'$.
The condition 
$\A$ 
holds, e.g., for {\it special 
subordinate 
Brownian motions}, a class of processes wider than the complete subordinate Brownian motions.
We should 
note that $V$ is defined implicitly but in the considered isotropic
setting it enjoys 
simple sharp estimates in terms of more elementary functions, namely
the L\'evy-Khintchine exponent $\psi$ of $X$ and the following Pruitt's function $h$ 
\cite{MR632968} (see \eqref{def:GKh} below for details),
\begin{align}
h(r)&=\frac{\sigma^2 d}{r^2}+\int\limits_{\Rd}\left(\frac{|z|^2}{r^2}\wedge 1\right)\nu(dz),\quad r>0.\label{dhbdh}
\end{align}
Namely, it follows from Proposition~\ref{ch1Vp} and \eqref{eqcpg} that for 
unimodal L\'evy processes with unbounded $\psi$ we have
\begin{equation}
h(r)\approx \psi(1/r)\approx1/V(r)^{2}
,\qquad r>0.
\end{equation}
On the other hand, the control of $V'$ is hard. For instance continuity and monotonicity of $V'$, although common, are open to conjectures. (We actually know that $V'$ may fail to be monotone for some unimodal L\'evy processes, see Remark~\ref{rmV}.)
For  complete subordinate Brownian motions good control 
results from the fact that $V'$ is completely monotone.
This sheds 
light on
the results 
obtained by
Z.~Chen, P.~Kim, R.~Song and Z.~Vondra\^{c}ek
(cf. \cite{
MR2923420,2012arXiv1212.3092K} and Section~\ref{sec:ex} below).
Our approach allows to lift this structure requirement that $X$ is a subordinate Brownian motion,
thanks to new ideas employing unimodality, scaling and (approximating) Dynkin's operator.

The basic object of interest in our study
is 
$\E^x\tau_{B_r}$,
the expected exit time from the ball $B_r$ centered at the origin and with radius $r>0$,
for arbitrary starting point $x\in \Rd$ of  
$X$ (for detailed definitions see Section~\ref{sec:prel}). When $x=0$, the classical result of Pruitt  \cite{MR632968} (see p.~954, Theorem 1 and (3.2) ibid.)
provides in our setting constants $c=c(d)$ and $C=C(d)$ such that 
\begin{equation} \label{Pruitt}
\frac{c}{h(r)}\le
\E^0 \tau_{B_r}\leq \frac{C}{h(r)}, \qquad r>0.
\end{equation}
Pruitt's estimate may be called {\it sharp}, meaning that
the ratio of its extreme sides 
is bounded.
One of our main contributions is the following inequality,
\begin{equation} \label{Pruitt1}
\frac{c}{\sqrt{h(r)h(r-|x|)} } \le
\E^x\tau_{B_r}\leq \frac{C}{\sqrt{h(r)h(r-|x|)}},\qquad x\in B(0,r),
\end{equation}
where $c=c(r,d, X)>0$  is non-increasing in $r$ and  $C=C(d)$. The estimate holds for unimodal L\'evy processes under 
condition $\A$ on $V'$.
The estimate is {\it sharp} up to the boundary of the ball.
As we note in Lemma~\ref{ExitTimeUpper}, the upper bound in \eqref{Pruitt1} easily follows from the one-dimensional case
\eqref{exitTimeOneDim}, cf. \cite{MR3007664}. 
The lower bound is much more delicate.
To the best of our knowledge the
lower bound was only known for complete subordinate Brownian motions
satisfying certain scaling conditions
(see Theorem 1.2 and Proposition 2.7 in \cite{2012arXiv1208.5112K}).
Our results cover in a uniform way isotropic stable process, relativistic stable process, sums of two independent 
isotropic stable processes (also with Gaussian component), geometric stable processes, variance gamma processes, conjugate to geometric stable processes \cite{MR2978140} and much more which could not be treated by previous methods.
The fact that  $c
$ in \eqref{Pruitt1} depends on $r$ is a
drawback if one needs to consider large $r$.
In many situations, however, we may actually choose $c$ independent of
$r$. For example if $X$ is a special subordinate Brownian motion,
then we have $c=c(d)$, which    follows  by combining Theorem \ref{Exit2} with Lemma \ref{specialConcave} below.
We conjecture that in the case of isotropic L\'evy processes, one can always choose $c$
depending only on $d$.
This is certainly true in the one-dimensional case, see \eqref{exitTimeOneDim}.
For $d\ge 2$ the conjecture is strongly supported by comparison of \eqref{Pruitt} and \eqref{Pruitt1}.

We test super- and subharmonicity by means of Dynkin's generator of $X$ in a way suggested by 
\cite{MR2320691}.
We also rely on our recent bounds for the
semigroups of weakly scaling unimodal L\'evy processes on the whole of $\Rd$
\cite{2013arXiv1305.0976B}, and results of T.~Grzywny \cite{2013arXiv1301.2441G}.
As we indicated above, delicate properties of $V$, indeed of $V'$, are used to prove \eqref{Pruitt1} by way of calculating Dynkin's operator
on functions defined with the help of $V$.
Fortunately, the resulting asymptotics is directly expressed by $V$, rather then by $V'$, and may also be described by means of the L\'evy-Kchintchine exponent $\psi$ or $h$, which we indeed do in \eqref{Pruitt1}. 
(Estimates expresses in terms of $h$ may be considered the most explicit, because $h$ is given by a direct integration without cancellations.)

On a general level our development rests on estimates for Dynkin-type generators acting on smooth test functions (Section~\ref{sec:prel}) and compositions of $V$ (Section~\ref{OwspanialaV}). This explains our restriction to $C^{1,1}$ open sets: we approximate them by translations and rotations of the half-space $\mathbb{H}=\{x\in \Rd: x_1>0\}$, and $V(x_1)$ is harmonic for $X$ on $\mathbb{H}$.
Noteworthy, the so-called boundary Harnack principle (BHP) for harmonic functions of $X$
is negligeable in our development; it is superseded 
in estimates by
the ubiquitous function $V$.  Barriers resulting from $V$ provide access to asymptotics of the expected exit time, survival probability, Green function, harmonic measure, distribution of the exit time and the heat kernel.
In fact, our estimates imply explicit decay rate for nonnegative harmonic functions near the boundary of $C^{1,1}$ open sets, see Proposition~\ref{prop:BHP}. Furthermore, in
\cite{BGR2013_3} we 
give applications to  heat kernels for the corresponding Dirichlet problem in $C^{1,1}$ open sets.
We also expect applications to Hardy-type inequalities, cf. \cite{MR856511}.

It would be of considerable interest to further extend our 
estimates to Markov processes with  isotropic L\'evy kernels $dy\mapsto \nu(x,dy)$ or to isotropic L\'evy processes with the L\'evy measure approximately unimodal  in the sense of  \eqref{generalize1}. Partial results in this direction are given in Corollary~\ref{genaral1}.
We like to note that the case of Lipschitz open sets apparently requires approach based on BHP 
and is bound to produce less explicit estimates. We refer the reader to \cite{MR2722789,MR2904138} for more information and bibliography on this subject. In this connection we like to note that BHP fails for non-convex open sets for the so-called truncated stable L\'evy processes \cite{MR2282263}.

The rest of the paper is composed as follows.
In Section~\ref{sec:prel} we estimate tails of $X_t$ and $X_{\tau_D}$ 
by means of $\E^x\tau_D$, $V$ or $h$.
In Lemma~\ref{Vestimate1}  and \ref{Vestimate2} of Section~\ref{OwspanialaV} we construct mildly super- and subharmonic functions for the ball and the complement of the ball, respectively.
In Section~\ref{secExit} we estimate the expected exit time:
Theorem~\ref{Exit2} provides \eqref{Pruitt1} and Theorem~\ref{Exit_C11} states (with more detail) the following  estimates of the expected exit time of unimodal L\'evy processes with unboundel L\'evy-Khintchine exponent from $C^{1,1}$ open bounded sets $D$, under mild conditions including $\A$,
$$
\E^x\tau_{D}\approx  V(\delta_{D}(x)),
\quad x\in \Rd.$$
In Section~\ref{sec:ScalingAndConsequences} we consider the case of transient $X$, and estimate the probability of ever hitting the ball from outside in, say, dimension $d\ge 3$, by using the estimates of T.~Grzywny \cite{2013arXiv1301.2441G} for potential kernel:
$ U(x)\le cV^2(|x|)/|x|^d$ for $x\in \R^d$, and for the 
the capacity of the ball: $\CAP(\overline{B_r})\approx r^d/V^2(r)$ for $r>0$.
In Section~\ref{Survival} under weak scaling conditions we estimate the survival probability:
$$
 \p^x(\tau_{D}>t)\approx  \frac{V(\delta_D(x))}{\sqrt{t}}\wedge 1,\qquad x\in \Rd, \quad 0<t\le C V(r_0)^2.
$$
Here $r_0$ is the $C^{1,1}$-localization radius of $D$. The result is new even for complete subordinate Brownian motions.
Further estimates and information are given as we proceed.

In Section~\ref{sec:conditionA} we discuss the role and validity of $\A$ and give specific examples of L\'evy processes manageable by our methods.
 Since $V(\delta_D(x))\approx \left[\psi(1/\delta_D(x))\right]^{-1/2}$, our estimates are often entirely explicit.

As we advance, the reader should observe the assumptions specified at the beginning of each section: as a rule they bind 
the statements of the results in that section.
Notably, a large part of our estimates, especially of the upper bounds, are valid under minimal assumptions including isotropy and, usually but not always (cf. Section~\ref{sec:prel}), unimodality of $X$. Scaling, unimodality, pure-jump character of $X$  and the  Harnack-type condition $\A$ on $V'$ are commonly assumed to prove matching lower bounds. 
We strive to make explicit the dependence of constants in our estimates on characteristics of $D$ and $X$. Some of the constants depend only on $d$ for all isotropic L\'evy processes, others depend on the assumption of unimodality, the parameters in the weak scaling and other analytic  properties of $X$ expressed through the L\'evy 
measure.
Good control of constants in estimates at scale $r>0$ necessitates the use of rather intrinsic quantities $\mathcal{I}(r)$ and $\mathcal{J}(r)$ introduced in Section~\ref{secExit}. Such control is especially important for the study of unbounded sets.

\section{Preliminaries}\label{sec:prel}
We write $f(x)\approx g(x)$ and say that $f$ and $g$ are {\it comparable} if $f, g\ge 0$ and there is a positive number $C$, called comparability constant,
such that $C^{-1}f(x)\le g(x)\le C f(x)$ for all  considered $x$.
We write $C=C(a,\ldots,z)$ to indicate that (constant) $C$
may be so chosen to depend only on $a,\ldots,z$. Constant may change values from place to place except for capitalized numbered constants ($C_1$, $C_2$ etc.), which are the same at each occurrence. 

We consider the Euclidean space $\R^d$ of arbitrary dimension $d\in \N$.
All sets, functions and measures considered below are assumed Borel.
Let $B(x,r)=\{y\in \Rd: |x-y|<r\}$,
the open ball with center at $x\in \Rd$ and radius $r>0$,
and  let $B_r=B(0,r)$. We denote by $\omega_d=2\pi^{d/2}/\Gamma(d/2)$ the surface measure of the unit sphere in $\Rd$. We also
consider {\it exterior} sets $B^c(x,r)=\(B(x,r)\)^c=\{y\in \Rd: |x-y|\ge r\}$,  $B^c_r=\(B(0,r)\)^c$ and $\overline{B}^c_r=\(\overline{B(0,r)}\)^c$.
For $D\subset \Rd$ we consider the distance to the complement of $D$:
$$\delta_D(x) =\dist(x,D^c)\,,\qquad x\in \Rd.$$

We say that $D$ is of class $C^{1,1}$ at scale $r$ if $r>0$, $D$ is open nonempty set in $\Rd$
and for every $Q\in \partial D$ there are balls
$B(x',r)\subset D$ and $B(x'',r)\subset D^c$ tangent at $Q$.
Thus, $B(x\rq{},r)$ and $B(x\rq{}\rq{},r)$ are the {\it inner}\/ and {\it outer}\/ balls at $Q$, respectively.
Estimates for $C^{1,1}$ open sets often rely on the inclusion $B(x\rq{},r)\subset D\subset B(x\rq{}\rq{},r)^c$, domain monotonicity of the considered quantities and on explicit calculations for the extreme sides of the inclusion.
If $D$ is $C^{1,1}$ at some unspecified scale (hence also at all smaller scales),
then we simply say $D$ is $C^{1,1}$.
 The $C^{1,1}$-{\it localization radius}, $$r_0=r_0(D)=\sup\{r: D \mbox{ is } C^{1,1} \mbox{ at scale } r\},$$
describes
the local geometry of such $D$, while the {\it diameter},
$${\rm diam}(D)=\sup\{|x-y|:\;x,y\in D\}\,,$$ depends on the
global geometry of $D$.
The ratio ${\rm diam}(D)/r_0(D)\geq 2$ is called the {\it distortion}\/ of $D$.
We remark that $C^{1,1}$ open sets may be defined by using local coordinates and Lipschitz condition on the gradient of the function defining their boundary (see, e.g., \cite[Section 2]{MR2286038}), hence the notation $C^{1,1}$. They can also be {\it localized} near the boundary without much changing the distortion \cite[Lemma~1]{MR2892584}.
Some of the comparability constants in our estimates depend on $D$ only through $d$ and the distortion of $D$.

We denote by $C_c(D)$ the class of the continuous functions on $\Rd$ with support in (arbitrary) open $D\subset \Rd$, and we let $C_0(D)$ denote the closure of $C_c(D)$ in the supremum norm.

A L\'evy process is a stochastic process $X=(X_t,\,t\ge 0)$ with values in $\Rd$, stochastically independent increments, c\'adl\'ag paths  and such that $\p(X(0)=0)=1$ \cite{MR1739520}.
We use $\p$ and $\E$ to denote the distribution and the expectation of $X$ on the space of c\'adl\'ag paths $\omega:[0,\infty)\to \Rd$, in fact $X$ may be considered as the canonical map: $X_t(\omega)=\omega(t)$ for $t\ge 0$.
In what follows, we shall use the Markovian setting for $X$, that is we define the distribution $\p^x$ and  the expectation $\E^x$ for the L\'evy process starting from arbitrary point $x\in \Rd$:
$\E^xF(X)=\E F(x+X)$ for Borel functions $F\ge0$ on paths. For $t\ge 0$, $x\in \Rd$, $f\in C_0(\Rd)$ we let $P_tf(x)=\E^x f(X_t)$, the semigroup of $X$.
We define
the time of the first exit of $X$ from (Borel) $D\subset \Rd$:
$$\tau_D=\inf\{t>0: \, X_t\notin D\}.$$
This random variable
gives rise to a number of important objects in the potential theory of $X$.
We shall focus on the expected exit time,
\begin{equation}\label{defsDk}
s_D(x)=\E^x \tau_D,\qquad x\in \Rd,
\end{equation}
and the survival probability
$$\p_x(\tau_D>t),\qquad x\in \Rd,\,t>0.$$
We shall also use the harmonic measure of $D$ for $X$ defined as
$$
\omega_D^x(A)=\p^x(X_{\tau_D}\in A), \qquad x\in \Rd, \quad A\subset \Rd.
$$
A real-valued function $f$ on $\Rd$ is called {\it harmonic} (for  $X$) on open $D\subset \Rd$ if for every open $U$ such that $\overline{U}$ is a compact subset of $D$, we have
\begin{equation}\label{harm_def}
f(x)=\E^x f(X_{\tau_U})=\int_{U^c}f(y)\omega^x_U(dy),\quad x\in U,
\end{equation}
and the integral is  absolutely convergent.
In particular, if $g$ is defined on $D^c$,
and $f(x)=\E^x g(X_{\tau_D})$ is absolutely convergent for $x\in D$, then $f$ is harmonic on $D$. This follows from the strong Markov property of $X$ \cite{MR1406564}. A function $f$ is called {\em regular} harmonic in $D$ if   \eqref{harm_def} holds for  $U=D$.

\subsection{Isotropic L\'evy processes}\label{iLp}
L\'evy measure is a (nonnegative Borel) measure concentrated on $\Rdz$ such that
\begin{equation}\label{wml}
\int_\Rd \left(|x|^2\wedge 1\right)\nu(dx)<\infty.
\end{equation}
We call measure on $\Rd$ {\it isotropic} if it is invariant upon linear isometries of $\Rd$ (i.e. symmetric if $d=1$).
A L\'evy process $X_t$ \cite{MR1739520} is called isotropic
if all its one-dimensional distributions $p_t(dx)$ are isotropic.
Isotropic L\'evy processes are characterized by L\'evy-Khintchine (characteristic) exponents of the form
\begin{equation}\label{characFun}\psi(\xi)=\sigma^2|\xi|^2+\int_\Rd \left(1- \cos \langle\xi,x\rangle\right) \nu(dx),\end{equation}
with isotropic
L\'evy measure
$\nu$ and $\sigma\ge 0$. To be specific, by the L\'evy-Kchintchine formula,
$$
\E\,e^{i\left<\xi, X_t\right>}=\int_\Rd e^{i\left<\xi,x\right>}p_t(dx)=e^{-t\psi(\xi)},\quad \xi\in\Rd.$$
Unless explicitly stated otherwise, in what follows we
assume that $X_t$ is an isotropic
L\'{e}vy process in $\Rd$  with
L\'evy measure $\nu$ and characteristic exponent $\psi\not\equiv 0$.
(We shall make additional assumptions in Sections~\ref{iLpim}, \ref{s:infty} and \ref{OwspanialaV}.)
Since
$\psi$ is a radial function,  we shall often write
$\psi(u)=\psi(x)$, where $x\in \Rd$ and $u=|x|\ge 0$.
For the first coordinate $X_t^1$ of $X_t$ we obtain the same function $\psi(u)$. Clearly, $\psi(0)=0$ and $\psi(u)>0$ for $u>0$.

For $r>0$ we define, after \cite{MR632968},
\begin{align}
\nonumber
K(r)&=\int\limits_{B_r}\frac{|z|^2}{r^2}\nu(dz), \quad L(r)= \nu\(B^c_r\),\\
h(r)&=\frac{\sigma^2 d}{r^2}+K(r)+L(r)=\frac{\sigma^2 d}{r^2}+\int\limits_{\Rd}\left(\frac{|z|^2}{r^2}\wedge 1\right)\nu(dz).\label{def:GKh}
\end{align}
We note that
$0
\le K(r)\leq h(r)<\infty$, $L(r)\ge 0$, $h$ is (strictly) positive and  decreasing, and $L$ is non-increasing.
The corresponding quantities for $X_t^1$, say $K_1(r)$, $L_1(r)$, $h_1(r)$, are given by the L\'evy  measure $\nu_1=\nu\circ x_1^{-1}$ on $\R$ \cite[Proposition~11.10]{MR1739520}, in particular
$$h_1(r)=\frac{\sigma^2}{r^2}+\int_\R \left(\frac{u^2}{r^2}\wedge 1\right)\nu_1(du)= \frac{\sigma^2}{r^2}+
\int\limits_{\Rd}\left(\frac{|z_1|^2}{r^2}\wedge 1\right)\nu(dz),\quad r>0.
$$
We see that
\begin{equation}\label{coh}
h_1(r)\le h(r)\le h_1(r)d,\qquad r>0.
\end{equation}

We 
shall make connections to the expected exit time of $X$ for general open sets $D\subset\Rd$.
By domain-monotonicity of exit times and Pruitt's estimate \eqref{Pruitt}, we have
\begin{equation}\label{RSP}
s_D(x)\leq
s_{B(x,{\rm diam}(D))}(x)\leq \frac{C}{h({\rm diam}(D))}<\infty.
\end{equation}
Our first lemma is a slight improvement
of  \cite[Lemma~3]{2013arXiv1301.2441G}.
\begin{lem}\label{L7}
If $r>0$ and $x\in B_{r/2}$,
then $\p^x\left(|X_{\tau_{D}}|\ge r\right)
\le
24\, h(r)\, \E^x\tau_{D}$.
\end{lem}
\begin{proof}
Let {$r>0$}. Let ${\cal A}$ be the generator of the semigroup of $X$ acting on $C_0(\Rd)$.
If $\phi\in C_{c}^2(\Rd)$, then $\phi$ is in the domain of ${\cal A}$. If $c\in \R$ and $f=c+\phi$, then
by Dynkin's formula,
\begin{eqnarray}\label{genDyn}
	\E^x \int_0^{\tau_D} {\cal A} \phi(X_s)ds=
\E^x\phi(X_{\tau_{D}})-\phi(x)=
\E^xf(X_{\tau_{D}})-f(x),\quad x\in \Rd,
\end{eqnarray}
and the generator may be calculated pointwise as
\begin{eqnarray*}
{\cal A}\phi(x)&=& \sigma^2\Delta f(x)+\int \left[f(x+z)-f(x)
-{\bf 1}_{|z|<1}\left<z,\nabla f(x)\right>\right]\nu(dz)=:{\cal A}f(x).
\end{eqnarray*}
Since
$\nu$ is symmetric, we can replace ${\bf 1}_{|z|<1}$ in the above equation by
${\bf 1}_{|z|<r}$.
We shall use
a function $g: \R^+ \mapsto [0, 1]$ such that $g(t)=0$ for $0\le t\le 1/2$, $g(t)=1$ for $t\ge 1$, and $\esssup_{t\in R^+} |g'(t)|$ and $\esssup_{t\in R^+} |g''(t)|$ are finite.
In fact, we initially let $g''=16$ on $(1/2,3/4)$ and $g''=-16$ on $(3/4,1)$, which gives  $\|g''\|_\infty=16$ and $\|g'\|_\infty=4$. We then have
\begin{eqnarray}\label{eq:25}
4\sup_{t\in R^+} |g'(t)| +\frac12\sup_{t\in R^+}|g''(t)| &=&24,\\
2(d-1)\sup_t|g'(t)|+ \sup_t|g''(t)|&=&8(d+1).\label{eq:16}
\end{eqnarray}
These will only slightly increase  as we modify $g''$ to be continuous. (Such modified $g\in C^2$ is used below.)
Denote
$$f_r(y)=g(|y|/r), \quad y\in \Rd.$$
We first consider $f_1$. Let $v,z\in \Rd$. There is a number $\theta$ between $|v|$ and $|v+z|$,
such that
\begin{eqnarray*}f_1(v+z)-f_1(v)
&=& g'(|v|)(|v+z|-|v|)+(1/2) g''(\theta)(|v+z|-|v|)^2\\
&=& g'(|v|)\frac{(|v+z|^2-|v|^2)}{|v+z|+|v|}+(1/2) g''(\theta)(|v+z|-|v|)^2\\
&=& g'(|v|)\frac{|z|^2+ 2\left<v,z\right>}{|v+z|+|v|}+(1/2) g''(\theta)(|v+z|-|v|)^2\\
&=& g'(|v|)\frac{ \left<v,z\right>}{|v|}+g'(|v|)\frac{ \left<v,z\right>}{|v|}\frac{|v|-|v+z|}{|v+z|+|v|}+g'(|v|)\frac{|z|^2}{|v+z|+|v|}\\&+&(1/2) g''(\theta)(|v+z|-|v|)^2.
\end{eqnarray*}
Since $g''=0$ on $B_{1/2}$, we have
$$\left|g'(|v|)\frac{ \left<v,z\right>}{|v|}\frac{|v|-|v+z|}{|v+z|+|v|}\right|\le |g'(|v|)|\frac{|z|^2}{|v+z|+|v|}\le 2|g'(|v|)||z|^2.$$
Also,
$$ \frac12g''(\theta)(|v+z|-|v|)^2\le \frac12|g''(\theta)||z|^2.$$
Since $$\left<z,\nabla f_1(v)\right>=  g'(|v|)\frac{ \left<v,z\right>}{|v|},$$
we
obtain
$$\left|f_1(v+z)-f_1(v)
-{\bf 1}_{|z|<1}\left<z,\nabla f_1(x)\right>\right|\le
\left(4\sup_t|g'(t)|+ \frac12\sup_t|g''(t)|\right)|z|^2.
$$
By changing variables we have
$$
|f_r(v+z)-f_r(v)
-{\bf 1}_{|z|<r}\left<z,\nabla f_r(v)\right>|
\le
\left(4\sup_t|g'(t)|+ \frac12\sup_t|g''(t)|\right)
|z/r|^2.$$
We also note that
$$|\Delta f_1(z)|=|(d-1)g'(|z|)/|z|+g''(|z|)|\leq 2(d-1)\sup_t|g'(t)|+ \sup_t|g''(t)|.$$
Applying
(\ref{genDyn}) to $f(y)=f_r(y)=g(|y|/r)$,
we get
\begin{eqnarray}\label{gen1}
\E^xf_r(X_{\tau_{D}})=\E^x \int_0^{\tau_D} {\cal A}f_r(X_s)ds
, \qquad |x|\le r/2.
\end{eqnarray}
By the preceding estimates,
\begin{eqnarray}
{\cal A}f_r(v)&=& \sigma^2\Delta f_r(v)+\int \left(f_r(v+z)-f_r(v) \nonumber
-{\bf 1}_{|z|<r}\left<z,\nabla f_r(x)\right>\right)\nu(dz)\\
&\le&\sigma^2\frac{\;2(d-1)\sup_t|g'(t)|+ \sup_t|g''(t)|}{r^2}\label{ogAr}\\
&&+\frac{4\sup_t|g'(t)|+ \frac12\sup_t|g''(t)|
}{r^2}\int_{|z|<r}|z|^2\nu(dz) +\nu(B^c_{r})
.\nonumber
\end{eqnarray}
Using $\p^x\left(|X_{\tau_{D}}|\ge r\right)\le \E^xf_r(X_{\tau_{D}})$,
(\ref{gen1}), \eqref{ogAr}, \eqref{eq:25}, \eqref{eq:16} and \eqref{def:GKh}, we get
the result.
\end{proof}

\begin{remark} The approach generalizes to other stopping times, e.g.
deterministic times $t>0$:
\begin{equation}\label{eq:ft}
\p^x\left(|X_t|\ge r\right)\le
24\, h(r)\,t, \quad r>0,\quad |x|\le r/2.
\end{equation}
\end{remark}

Recall that
$p_t(dx)$ has no atoms if and only if $\psi$ is unbounded (if and only if $\nu(\Rd)=\infty$ or $\sigma>0$) \cite[Theorem 30.10]{MR1739520}.
In fact, if $\sigma>0$ or if $d\ge 2$ and $\nu(\Rd)=\infty$, then $(p_t,t>0)$ have lower semicontinuous density functions
\cite[
(4.6)]{MR0267643}.
We further note that the {\it resolvent measures}
$$
A\mapsto \int_0^\infty p_t(A)e^{-q t} dt, \qquad q>0,
$$
are absolutely continuous if and only if $p_t$, $t>0$, are absolutely continuous. This consequence of symmetry of $p_t$ is proved in \cite[Theorem~6]{MR0341626}, see also \cite[Remark 41.13]{MR1739520}.

\subsection{Isotropic L\'evy processes with unbounded characteristic exponent}\label{iLpim}
Unless explicitly stated otherwise, in what follows $X$ is an isotropic L\'evy process with unbounded L\'evy-Kchintchine exponent $\psi$.

Let $M_t=\sup_{s\le t}X^1_s$ and let $L_t$ be the local time {at $0$} for
{$M_t-X^1_t$}, the first coordinate of $X$ reflected at the supremum  (\cite{MR0400406},\cite{MR1406564}).
We consider its right-continuous inverse, $L^{-1}_s$,
called the ascending ladder time process for $X^1_t$.
We also define the ascending ladder-height process,
$H_s = X_{L^{-1}_s}^{1} = M_{L^{-1}_s}$.
The pair $(L^{-1}_t,H_t)$ is a two-dimensional subordinator (\cite{MR0400406},\cite{MR1406564}).
In fact, since $X_t^{1}$ is symmetric and has infinite L\'evy measure or nonzero Gaussian part, by \cite[Corollary~9.7]{MR0400406}, the Laplace exponent of $(L^{-1}_t,H_t)$ is
$$
\log \left(\E \exp[-\tau L_t^{-1}-\xi H_t]\right)=c_+\exp\left\{
\frac1{\pi}\int_0^\infty
\frac{\log\left[\tau+\psi(\theta\xi)\right]}{1+\theta^2}
d\theta\right\},\qquad \tau,\xi\ge0,
$$
In what follows we 
let $c_+=1$, thus
normalizing
the local time $L$ \cite{MR0400406}. In particular,
$L^{-1}_s$  is then
the standard $1/2$-stable subordinator (see also \cite[(4.4.1)]{MR2320889}),
and the Laplace exponent of $H_t$ is
\begin{equation}\label{kappa}
 \kappa(\xi)=
\log \left(\E \exp[-\xi H_t]\right)= \exp\left\{\frac{1}{\pi} \int_0^\infty \frac{ \log {\psi}(\xi\zeta)}{1 + \zeta^2} \, d\zeta\right\}, \quad \xi\ge 0.
\end{equation}
The renewal function $V$ of the ascending ladder-height process $H$ is defined as
\begin{equation}\label{e:defV}
V(x) = \int_0^{\infty}\p(H_s \le x)ds, \quad
x {\in \Rd}.
\end{equation}
Thus, $V(x)=0$ if $x<0$ and $V$ is non-decreasing. It is also well known that $V$ is subadditive,
\begin{equation}\label{subad}
 V(x+y)\le V(x)+V(y), \quad x,y \in \R,
\end{equation}
 and $V(\infty)=\infty$.
Both $V$ and its derivative $V'$ play a crucial role in our
development.
They were studied by Silverstein as $g$ and $\psi$ in \cite{MR573292}, see (1.8) and Theorem~2 ibid., respectively.
If  resolvent measures of $X^1_t$ are absolutely continuous, then it follows from \cite[Theorem 2]{MR573292}
that
$V(x)$ is absolutely continuous and harmonic on $(0,\infty)$ for the process $X_t^1$, in fact, $V
$ is invariant for the process $X_t^1$ killed on exiting $(0,\infty)$. Also,
$V^\prime$
is a
positive harmonic function for
$X_t^1$ on $(0,\infty)$, hence $V$ is actually (strictly) increasing.
Notably, the definition of $V$ is rather implicit and the study of $V$ poses problems.
In fact, we shall shortly present sharp estimates of $V$ by means of  (simpler) functions $\psi$ and $h$, but decay properties of $V'$ are more delicate and they are not yet fully understood.
Under structure assumptions satisfied for complete subordinate Brownian motions, $V'$ is monotone, in fact completely monotone (cf. Lemma~\ref{specialConcave}). This circumstance stimulated
much of
the progress made in \cite{
MR2923420,2012arXiv1212.3092K}.
The methods presented below in this paper address more general situations, e.g. when the L\'evy-Khintchine exponent $\psi$ has {\it weak scaling} or when $X$ has a nonzero Gaussian part  (see  Section~\ref{ss:cA}).

By \cite[Corollary 4 and Theorem 3]{MR2320889}
and \cite[Remark 3.3 (iv)]{MR2978140} the following result holds.
\begin{lem}\label{kappaDrift}
We have
$\lim_{\xi\to \infty}\kappa(\xi)/\xi=\sigma$. Furthermore, if $\sigma>0$, then $V'$ is continuous,  positive and bounded by $\lim_{t\to0^+}V'(t)=\sigma^{-1}$.
In fact $V'$ is bounded if and only if $\sigma>0$.
\end{lem}

As we indicated in Section~\ref{sec:i}, estimates
of $\E^x\tau_{B_{r}}$, the expected exit time  from the ball play an important role in this paper.
The upper bound \eqref{Pruitt}, sharp at the center of the ball, was given by Pruitt in \cite[p.~954]{MR632968}. It was later generalized to more general Markov processes by Schilling in \cite[Remark 4.8]{MR1664705}.
For every symmetric L\'evy process $X$ on $\R^1$ with unbounded L\'evy-Kchintchine exponent $\psi$, the following bound with absolute constant $C_0>0$ follows from \cite[Proposition 3.5]{MR3007664} by Grzywny and Ryznar {and subadditivity of $V$},
\begin{equation}\label{exitTimeOneDim}
{C_0}V(r)V(r-|x|) \le \E^x\tau_{(-r,r)}\le{2}V(r)V(r-|x|) , \qquad
x\in \R,\quad r>0. \end{equation}
In  Section~\ref{secExit} we establish a similar {\it comparability} result in arbitrary dimension under appropriate conditions on $X$.
The upper bound is, however, simpler, and we can give it immediately.
\begin{lem}\label{ExitTimeUpper}
For all $r>0$ and $x\in \Rd$ we have $E^x\tau_{B_r}\leq 2 V(r)V(r-|x|)$.
\end{lem}
\begin{proof}Since $X$ is isotropic with unbounded L\'evy-Kchintchine exponent $\psi$, by Blumenthal's 0-1 law we have $\tau_{B_r}=0$ $\p^x$-a.s. for all $x\in B^c_r$. Hence, it remains to prove the claim for $x\in B_r$. If $\tau=\inf\{t>0: |X^1_t|>r\}$, then domain-monotonicity of the exit times and \cite[Proposition~3.5]{MR3007664} yield
$E^x\tau_{B_r}\le \E^x\tau\le V(r-|x_{1}
|)V(2r)$. By
\eqref{subad} and rotations we obtain the claim.
\end{proof}

We define the maximal characteristic function $\psi^*(u):= {\sup_{0\le s\leq {u}}\psi(s)}$, where $u\ge 0$.
\begin{prop}\label{ch1Vp}The constants in the following comparisons depend only on the dimension,
\begin{equation}\label{cVh1pgstare}
h(r)\approx h_1(r)\approx \psi^*(1/r)\approx\left[V(r)\right]^{-2}
,\quad r>0.
\end{equation}
\end{prop}
\begin{proof}
We shall see that all of the comparisons are absolute, except for
the first comparison in \eqref{cVh1pgstare}, which depends
on the dimension via \eqref{coh}.
Let $r>0$. {Since}
$X^1$
{is}
symmetric,
\begin{equation}\label{happrox}
h_1(r)\approx \psi^*(1/r),
\end{equation} 
see \cite[Corollary 1]{2013arXiv1301.2441G}.
Let $r>0$ and $\tau_r$ be the time of the first exit of $X^1_t$ from the interval $(-r,r)$. By  \eqref{exitTimeOneDim} and \cite[p.~954]{MR632968}
 (see also \cite[Remark 4.8]{MR1664705}), we have $V^2(r)\approx \E^0\tau_r \approx1/h_1(r)$.
\end{proof}

\begin{lem}\label{LimitV(t)/t}We have
$\lim_{t\to 0^+}t/V(t)=\sigma.$
\end{lem}
\begin{proof}
By  Proposition \ref{ch1Vp} and the dominated convergence theorem,
$$\frac{t^2}{V^2(t)}\approx  t^2h_1(t)=  \sigma^2+\int_{\R^d} \(t^2\wedge |z_1|^2\) \nu(dz)\to \sigma^2 \quad \text{ as $t\to 0$}.$$
This ends the proof when $X$ is pure-jump. If $\sigma>0$, then we use Lemma \ref{kappaDrift}.
\end{proof}

The next result on survival probability was
known before
in the situation when
$\psi(r)$ and ${r^2}/{\psi(r)}$ are
non-decreasing
in $r\in (0, \infty)$, see \cite[Theorem 4.6]{KMR}.
\begin{prop}\label{lalfline}
For every symmetric L\'evy process in $\R$ which is not compound Poisson,
\begin{equation}\label{halfline}
\p^x(\tau_{(0, \infty)}\ge t)\approx 1\wedge \frac{1}{\sqrt{t\psi^*(1/x)}},\quad t,x>0,\end{equation}
and the comparability constant is absolute.
\end{prop}
\begin{proof}
Considering that
$L^{-1}_s$  is a $1/2$-stable subordinator, from
\cite[Theorem 3.1]{KMR} we see that
\begin{equation}
\label{half}
\p^x(\tau_{(0, \infty)}\ge t)\approx 1\wedge \frac{V(x)}{\sqrt t}, \qquad {t,\, x >0.}
\end{equation}
The result now obtains from \eqref{half} and Proposition \ref{ch1Vp}.
\end{proof}

From \eqref{cVh1pgstare} and definitions of $L_1$, $L$ and $h$, we derive
the following inequality,
\begin{equation}\label{tails}
L_1(r)\le
L(r)\le h(r)\le c/[V(r)]^2, \qquad r>0.
\end{equation}

\begin{lem}\label{upper1}
There {is} $C_1=C_1(d)$
such that
{if $r>0$, $D\subset B_r$ and $x\in {D\cap} B_{r/2}$, then}
  \begin{align}\label{eq:2}
  \p^x\left(|X_t|\ge r\right)&\le C_1 \frac{ t}{V^2(r)}, \quad t>0,\\
\label{eq:l}\p^x\left(|X_{\tau_{D}}|\ge r\right)
&\le C_1 \frac{ \E^x\tau_{D}}{V^2(r)}, \quad \text{and}\\
\label{eq:2a} \E^x\tau_{B_{r}}&\ge
V^2(r)/C_1.
\end{align}
\end{lem}
\begin{proof}
Lemma~\ref{L7}, Proposition~\ref{ch1Vp} and \eqref{eq:ft} give \eqref{eq:2}  and \eqref{eq:l}, which yield
\eqref{eq:2a}.
\end{proof}

\begin{cor}\label{kula}
$C_2=C_2(d)$ and $C_3=C_3(d)$ exist such that for $t,r> 0$ and $|x|\le r/2$,
$$\p^x(\tau_{B_r}\le t  )\le C_2\frac t{V^2(r)},$$
and
$$\p^x(\tau_{B_r}>C_3V^2(r) )\ge 1/2.$$
\end{cor}
\begin{proof} Observe that for $|x|\leq r/2 $,
$$\p^x(\tau_{B_r}\le t  )\le P^0(\tau_{B_{r/2}}\le t  ).$$
By L\'{e}vy's inequality and (\ref{eq:2}) we obtain the first claim with $C_2=8C_1$, because
$$\p^0(\tau_{B_{r/2}}\le t  )= \p^0( \sup_{s\le t}|X_s|\ge r/2)\le 2 \p^0( |X_t|\ge r/2)\le 2C_1\frac t{V^2(r/2)}\le 8C_1\frac t{V^2(r)}.$$
Taking $t= V^2(r)/(16C_1)$ we prove the second claim with $C_3=(16 C_1)^{-1}$.
\end{proof}

We observe the following regularity of the expected exit time.
\begin{lem}\label{lcont}
If the resolvent measures of $X$ are absolutely continuous and the
open bounded set $D\subset \Rd$ has the outer cone property, then $s_D\in C_0(D)$.
\end{lem}
\begin{proof}
Recall that $s_D$ is bounded. We also have $s_D(x)=0$ for $x\in D^c$. Indeed, for $x\in \partial D$, by Blumenthal's 0-1 law we have
$\tau_D=0$ $\p^x$-a.s., because $X$ is isotropic with unbounded L\'evy-Kchintchine exponent $\psi$ and $D$ has the outer cone property.
Due to \cite[Theorem 6]{MR0341626} and \cite[Lemma 2.1]{MR531166}, $X$ is strong Feller. Hence,  for each $t>0$, $x\mapsto \p^x(\tau_D>t)$ is upper semicontinuous \cite[Proposition 4.4.1, p. 163]{MR2152573}.
Therefore $s_D(x)=\int_0^\infty \p^x(\tau_D>t)dt$ is also upper semi-continuous.
In consequence, $s_D(x)\to 0$ as $\delta_D(x)\to 0$, and so $s_D$ is continuous at $\partial D$.
To prove continuity of $s_D$ on $D$,
we let $D\ni z\to x\in D$, and denote
$$D'=D-(z-x),
\quad U=D\cap D',
\quad R=D\setminus U.
$$
We have $s_D(x)=s_U(x)+\int_R s_D(y)\omega^x_U(dy)$ and $s_D(z)=s_{D'}(x)\ge s_U(x)$,
thus
$s_D(z)\ge s_D(x)-\int_{R} s_D(y)\omega^x_U(dy)\to s_D(x)$,
because
if $y\in R$,
then $\delta_D(y)\le |z-x|$ and $s_D(y)$ is small.
We see that $s_D$ is lower semi-continuous on $D$, hence continuous in $D$, in fact on $\Rd$.
\end{proof}
\begin{remark}\label{rem:ac}
The resolvent measures are absolutely continuous in dimensions bigger than one, hence $s_D\in C_0(D)$ if $D$ is an open bounded set with the outer cone property in $\R^d$ and $d\ge 2$.
\end{remark}

\subsection{Isotropic
	absolutely continuous L\'{e}vy measure}\label{s:infty}
In what follows, unless stated otherwise, we assume that $X$ is an isotropic  L\'evy process in $\Rd$ with the L\'evy measure $\nu(dx)=\nu(x)dx$ and unbounded L\'evy-Kchintchine exponent $\psi$.
 In particular, $X$ is symmetric, not compound Poisson, has absolute continuous distribution for all $t>0$ and absolutely continuous resolvent measures. Indeed, the case of $d\ge 2$ was discussed in Section \ref{iLp} and  Remark~\ref{rem:ac}, and  for $d=1$ we invoke  \cite[Theorem~1 (i)(ii)]{MR0182061}). We may assume that the density functions $x\mapsto p_t(x)$ are lower-semicontinuous for every $t>0$, see \cite[Theorem~2.2]{MR531166}.

The transition
density of the process $X$ {\it killed off} $D$  is defined by Hunt's formula,
$$
p_D(t,x,y)
=p(t,x,y)-\E^x\big[p(t-\tau_D,X_{\tau_D},y);\tau_D<t\big],
\qquad t>0, \, x,y\in \Rd.
$$
We call $p_D$ the Dirichlet heat kernel of $X$ on $D$.
The Green function of $D$ for $X$ is defined as
$$
G_D(x,y)=\int_0^\infty p_D(t,x,y)dt.
$$
Here is a connection between the main objects of our study,
\begin{equation}\label{defsD}
s_D(x)=\E^x \tau_D=\int_\Rd G_D(x,y)dy=\int_0^\infty \p^x(\tau_D>t)dt.
\end{equation}
If $x \in D$, then the $\p^x$-distribution
of $(\tau_D,X_{\tau_D-},X_{\tau_D})$ restricted to $X_{\tau_D-}\neq X_{\tau_D}$ is given by
the following density function
\cite{MR0142153},
\begin{equation}\label{Ikeda-Watanabe} (0,\infty)\times D\times \left(\overline{D}\right)^c\ni(s, u, z) \mapsto\nu(z -u) p_D(s, x, u).\end{equation}
Integrating against $ds$, $du$ and/or $dz$
 gives marginal distributions.
For instance, if $x\in D$ and $\p^x(X_{\tau_D-}\in \partial D)=0$, then
\begin{equation}\label{Ikeda-Watanabe2}
\p^x(X_{\tau_D}\in dz)=\(\int_D G_D(x,u) \nu(z-u) du\)dz \quad \text{ on } \;(\overline{ D})^c.
\end{equation}
Identities resulting from \eqref{Ikeda-Watanabe}
are called Ikeda-Watanabe formulae for $X$. Noteworthy, they allow for intuitive interpretations in terms of the expected occupation time measures $p_D(s,x,u)du$ and $G_D(x,u)du$, and in terms of the measure of the intensity of jumps, $\nu(z-u)dz$, cf. \cite[p.~17]{MR2569321}.

\section{
Construction of barriers for unimodal L\'evy processes}\label{OwspanialaV}

A measure on $\Rd$ is called isotropic
unimodal, in short, unimodal, if it is absolutely continuous on $\Rd\setminus \{0\}$
with a radial non-increasing density function (such measures may have an atom at the origin). A L\'evy process $X_t$ is called (isotropic) unimodal if all its one-dimensional distributions $p_t(dx)$ are
unimodal.
Unimodal
L\'evy processes are characterized in \cite{MR705619}
by isotropic unimodal L\'evy measures
$\nu(dx)=\nu(x)dx=\nu(|x|)dx$.
The distribution 
of $X_t$ under $\E=\E^0$ has a radial nonincreasing density $p_t(x)$ on $\Rdz$, and atom at the origin, with mass $\exp[-t\nu(\Rd)]$
(no atom if $\psi$ is unbounded, i.e. if $\sigma>0$ or $\nu(\Rd)=\infty$).
We refer to \cite{2013arXiv1305.0976B} for additional discussion.
Unless explicitly stated otherwise, in what follows we always assume that $X$ is a unimodal L\'{e}vy process in $\Rd$ with unbounded L\'evy-Kchintchine exponent $\psi$. Recall that  by  \cite{2013arXiv1305.0976B},
\begin{equation}\label{eqcpg}
\psi(u)\le \psi^*(u)\leq \pi^2\, \psi(u) \quad \text{ for } \quad u\ge0.
\end{equation}

For $f:\Rd\to \R$, $t>0$ and $x\in \Rd$ we consider the (approximating) Dynkin operator,
$${\cal{A}}_t f(x)= \frac{\E^xf(X_{\tau_{B(x,t)}})-f(x)}{\E^x\tau_{B(x,t)}},$$
whenever $\E^xf(X_{\tau_{B(x,t)}})$ is well defined.  For instance, if $s_D(x)=\E^x \tau_D$ and $0<t\le\delta_D(x)$, then by the strong Markov property, $s_D(x)=s_{B(x,t)}(x)+\E^x s_D(X_{\tau_{B(x,t)}})$, and so
\begin{equation}\label{sx-1}
{\cal{A}}_t s_D(x)=-1.
\end{equation}
By a similar argument, if $f$ is harmonic on $D$, $x\in D$ and $0<t<\delta_D(x)$, then
${\cal{A}}_t f(x)= 0$, by the  (harmonic) mean-value property.
In particular,
let $\H=\{x=(x_1,\ldots,x_d)\in \Rd:\; x_1>0\}$ and $V_1(x)=V(x_1)$. Since $V$ is harmonic on $(0,\infty)\subset \RR$ for $X_1$, $V_1$ is harmonic in $\H$ for $X$ and so
${\cal{A}}_t V_{1}(x)=0$, if $0<t<\delta_\H(x)$. (This is the main reason why $V$ is relevant for construction of barriers for $C^{1,1}$ sets in $\Rd$.)
We also observe the following minimum principle: if $x$ is a point in $\Rd$ and $f(x)=\inf_{y\in \Rd}{f(y)}$, then
${\cal{A}}_t f(x)\ge 0$ for every $t>0$.
\begin{cor}
\label{max_p}
If ${\cal{A}}_t f(x)<0$ for some $t>0$, then $f(x)>\inf_{y\in \Rd}f(y)$.
\end{cor}
\begin{lem}\label{lc0}
If $f\in C_0(D)$ and for every $x\in D$ there is $t>0$ such that ${\cal{A}}_t f(x)<0$, then $f\ge 0$ on $\Rd$.
\end{lem}
\begin{proof}
Since $f$ attains its infimum on $\Rd$, but not on $D$ (cf. Corollary~\ref{max_p}),
we have $f\ge 0$.
\end{proof}
	
We make a simple observation on local regularity of harmonic functions,
motivated by\cite[proof of Lemma~6]{MR2365478} (see \cite{MR3065312, MR2180302} for more more in this direction).
\begin{lem}\label{lem:cfh}
Let $X$ be an isotropic L\'{e}vy process with absolutely continuous L\'{e}vy measure.
If $g$ is bounded on $\Rd$ and harmonic on open $D\subset \Rd$, then $g$ is continuous on $D$.
\end{lem}
\begin{proof}
For $r>0$, let $\omega_r(dy)=\p^0(X(\tau_{B_r})\in dy)$. Note that $g(x)=\int_\Rd g(y+x)\omega_r(dy)$ if $0<r<\delta_D(x)$. By isotropy and Ikeda-Watanabe formula,
$\omega_r(dy)=c_r\sigma_r(dy)+f_r(y)dy$, where
$\sigma_r$ is the normalized spherical measure on $\partial B_r$, $0\le c_r\le 1$, and
$$
f_r(y)=
\left\{\begin{array}{ll}
\int_{B_r} G_{B_r}(0,v)\nu(y-v)dv,
& \text{ if } |y|>r,\\
0& \text{ else. }
\end{array}\right.
$$
Let $\rho>0$. Note that the measure $\int_\rho^{2\rho}c_r\sigma_r(A) dr$ has
density function $\omega_d^{-1}{\bf 1}_{\rho<|x|<2\rho}\;|x|^{1-d}\,c_{|x|}$. Therefore
$\Omega_\rho(A)={\rho^{-1}}\int_\rho^{2\rho}\omega_r(A) dr$ is absolutely continuous,
with density function denoted $F_\rho$.
We have
$g(x)=\int_\Rd g(y+x)\Omega_\rho(dy)=\int_\Rd g(y)F_\rho(y-x){dy}$ if $\delta_D(x)>2\rho$.
So, locally on $D$, $g$ is a convolution of a bounded function with an integrable function, so it
is continuous on $D$.
\end{proof}

We shall approximate harmonic functions of $X$
on the ball and the complement of the ball.
To this end we first
estimate a number of auxiliary integrals.
To motivate the first estimate we note that the definition \eqref{def:GKh} of $h$ allows for detailed study. For instance, $h$ and $h'$ are monotone. In \eqref{hprim} we make another important quantitative observation in this direction.
\begin{prop}\label{Vintegral}
There is
$C=C(d)$ such that
\begin{eqnarray*}
 \int_{0}^r  {V(\rho) } \rho^d\nu(\rho)d \rho&\le& C \frac{r}{V(r)}
,\quad r>0.
\end{eqnarray*}
\end{prop}
\begin{proof}
{Recall} that  $K(u)=\omega_du^{-2}\int^{u}_0 \rho^{d+1}\nu(\rho)d\rho$, $L(u)=\omega_d\int_{u}^\infty \rho^{d-1}\nu(\rho)d\rho$, and $h(u)=K(u)+L(u)+u^{-2}\sigma^2 d$.
Since $\nu$ is
non-increasing, hence  a.e. continuous,
for a.e. $u\in \R$ we have
\begin{equation}\label{hprim}
h'(u)=-2u^{-1}K(u)+\omega_du^{d-1}\nu(u) - \omega_du^{d-1}\nu(u)-2u^{-3}\sigma^2 d=-2u^{-1}\left(K(u)+u^{-2}\sigma^2 d\right).
\end{equation}
Also,
\begin{equation}\label{VInt1}
\int_{0}^r  {V(\rho) } \rho^d\nu(\rho)d \rho \leq c_1 \int_0^{r/2} V(u)L(u) d u,
\end{equation}
because
\begin{eqnarray*}\rho^d\nu(\rho)&=&\frac{d2^d}{2^d-1}\int^{\rho}_{\rho/2}u^{d-1}\nu(\rho)du
\leq \frac{d2^d}{2^d-1}\int^{\rho}_{\rho/2}u^{d-1}\nu(u)du\leq  \frac{d2^d}{\omega_d(2^d-1)}L(\rho/2).
\end{eqnarray*}
By \eqref{coh} and \eqref{cVh1pgstare}, $V(u)\approx h^{-1/2}(u)$, and so \eqref{hprim} yields
$$V(u)L(u)\approx h^{-1/2}(u)(h(u)-K(u)-u^{-2}\sigma^2 d)=h^{-1/2}(u)(h(u)+\frac{u}{2}h'(u))=(uh^{1/2}(u))' \quad a.e.$$
From this and \eqref{VInt1} we obtain the result
\begin{eqnarray*}
\int_{0}^r  {V(\rho) } \rho^d\nu(\rho)d \rho&\leq&c_2  rh^{1/2}(r)\approx \frac{r}{V(r)}.
\end{eqnarray*}
\end{proof}

  \begin{lem} \label{VIntegralEstimate}
There exists a constant $C=C(d)$, such that for
$0<x<r$,
$$\int_0^{r}V'(y/2)\int_{ |y-x|}^r  \rho^d\nu(\rho/2)d \rho  dy\le C\frac{r}{V(r)}
,$$
and
$$
\int_0^{r} V'(y/2) {|y-x|^{d+1}} \nu(|y-x|/2) dy\le
C\frac{r}{V(r)}.$$
   \end{lem}
  \begin{proof} Since $\nu$ is decreasing,
we have
$${|y-x|^{d+1}} \nu(|y-x|{/2})\le {2(d+1)}\int_{ |y-x|{/2}}^{|y-x|}  \rho^d\nu(\rho/2)d\rho
,$$
hence
\begin{eqnarray*}
\int_0^{r} V'(y/2) {|y-x|^{d+1}} \nu(|y-x|{/2}) dy&\le&  {2(d+1)}\int_0^{r}V'(y/2)\int_{ |y-x|{/2}}^{{r}}  \rho^d\nu(\rho/2)d \rho  dy.
\end{eqnarray*}
To completely prove the lemma it is enough to estimate the latter integral. It equals
\begin{align*}
&2\int_0^r \rho^d \nu(\rho/2)\int_{(x/2-\rho)\vee 0}^{(x/2+\rho)\wedge r/2}V'(z)dzd\rho
\le 2\int_0^r \rho^d \nu(\rho/2)[V(x/2+\rho)-V(x/2-\rho)]d\rho\\
&
\le
4\int_0^r \rho^d \nu(\rho/2)V(\rho) d\rho\le cr/V(r),
\end{align*}
where we used subadditivity (\ref{subad}) of $V$ on $\R$ and Proposition~\ref{Vintegral}.
 \end{proof}

Recall that $V>0$ and $V'>0$ on $(0,\infty)$.
\begin{definition}
We say that
condition $\A$ holds if for every $r>0$  there is $H_r\geq 1$ such that
\begin{equation}\label{HR}
V(z)-V(y)\le H_r \,V^\prime(x)(z-y)\quad \text{whenever}\quad {0<}x{\le y\le z\le}5x\leq5r.
\end{equation}
We say that  $\As$ holds if  
$H_\infty=\sup_{r>0}H_r<\infty$.
\end{definition}
We consider $\A$ and $\As$ as 
versions of Harnack inequality because $\A$ 
is  
implied by
the following property:
\begin{equation}
\sup_{x\leq r,y\in[x,5x]} V'(y)\le H_r \inf_{x\leq r,y\in[x,5x]}V'(y), \qquad r>0.
\end{equation}
Both conditions control relative growth of $V$. If $\A$ holds, then
we may and do chose  $H_r$
non-decreasing in $r$.
Each of the following situations imply $\A$:
\begin{itemize}
\item[1.] $X$ is a subordinate Brownian motion governed by a special subordinator (see Lemma~\ref{specialConcave}).
\item[2.]  $d\ge 3$ and the characteristic exponent of $X$ satisfies WLSC (see \eqref{eq:LSC} and Lemma \ref{lem:MHR}).
\item[3.] $ d\geq 1$ and the characteristic exponent of $X$ satisfies WLSC and WUSC (see \eqref{eq:USC} and Lemma \ref{rem:MHR}).
\item[4.] $\sigma>0$ in \eqref{characFun} (see Lemma  
\ref{CondAsigma1}).
\end{itemize}
A more detailed discussion of
$\A $ and further examples are given in Section \ref{sec:conditionA}.

The following Lemma~~\ref{Vestimate1} and Lemma~\ref{Vestimate2} are the main results of this section. They exhibit nonnegative functions which are {\it superharmonic} (hence barriers) or {\it subharmonic} near the boundary of the ball, inside or outside of the ball, respectively. The functions are obtained by
composing $V$ with the distance to the complement of the ball or to the ball, respectively.
Super- and subharmonicity are defined by the left-hand side inequality in \eqref{ooD} and \eqref{subV}, respectively. The super- and subharmonicity of the considered functions are relatively mild as we have good control via the right-hand sides of these inequalities (see the proof of Theorem~\ref{Exit2} for an application).
In comparison with previous developments, it is
the use of Dynkin's operator that allows for calculations
which only minimally depend on the differential regularity of $V$. (The dependence on $V'$ is via the mean value type inequality $\A$.)

\begin{lem}\label{Vestimate1} Assume that 
$\A $ holds
or $d=1$.
Let $x_0\in \Rd$, $r>0$ and  $g(x) = V(\delta_{B(x_0,r)}(x))$. There is a constant $C_5=C_5(d)$
such that
\begin{equation}\label{ooD}
{0\le} {\limsup}_{t\to 0}\,\big[-\!\!{\cal{A}}_t g(x)\big]\le \frac {C_5\, H_r}{V(r)} \qquad \text{if }\; {0<}\delta_{B(x_0,r)}(x)< r/4.
\end{equation}
 \end{lem}
\begin{center}

\begin{figure}[h]

\caption{The settings for the proofs of Lemma~\ref{Vestimate1} (right) and Lemma~\ref{Vestimate2} (left).}
\vspace{10pt}

\begin{center}
\begin{tikzpicture}[scale=5]  

\draw[gray](0,0) coordinate (p) circle (0.36);
\draw[gray] (p)  circle (0.32);
\draw[gray] (p)  circle (0.28);
\draw[gray] (p)  circle (0.24);
\draw[gray] (p)  circle (0.16);
\draw[gray] (p)  circle (0.12);
\draw[gray] (p)  circle (0.08);
\draw[gray] (p)  circle (0.04);

\fill[white!100, opacity=1] (p) +(0.4,0)  circle (0.4);
\fill[white!100, opacity=1] (p) +(-0.8,0)  circle (0.8);

\draw (p) +(-0.8,0) coordinate(cdd) circle (0.8);
\draw (p) +(0.4,0) coordinate(cun) circle (0.4);
\draw (p) +(0.8,0) coordinate(cud)   circle (0.8);
\draw (p)  circle (0.4);
\draw (p)  circle (0.2);

\draw (p) +(0.13,0)  coordinate (c) circle (0.025);
\fill (c) circle (0.007);
\fill   (p) circle (0.01);
\draw (p) +(-0.13,0)  coordinate (c1) circle (0.025);
\fill (c1)   circle (0.007);
\draw (p)+(0,-0.87) -- (p); \draw (p)+(0,0.87) -- (p);

\end{tikzpicture}

\end{center}
\end{figure}
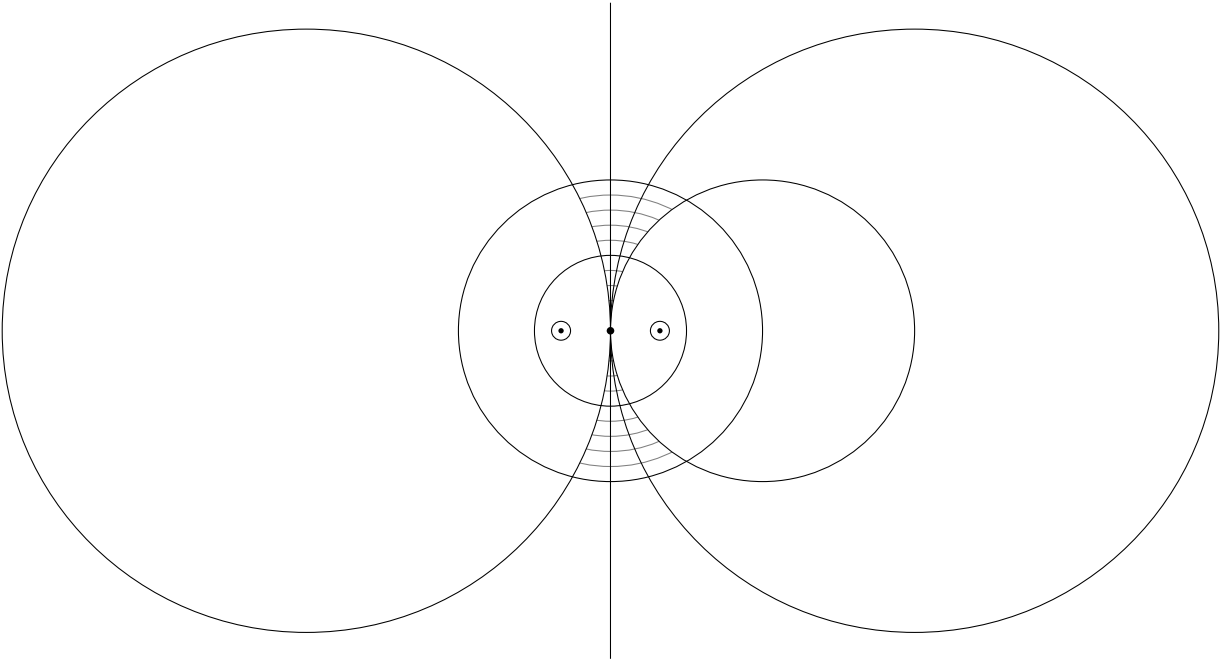
\end{center}

\begin{proof}
In what follows we use the notation $y=(\tilde{y}, y_d)$, where $y=(y_1,\ldots,y_d)\in \Rd$ and $\tilde{y}=(y_1,\ldots,y_{d-1})$.
Without loosing generality we may consider
\begin{equation}\label{x0}
x_0=(0,r)\quad \text{and} \quad  x=(0, x_d),\quad \text{where $0<4t<x_d<r/4$,}
\end{equation}
as shown on Figure~1 (in dimension $d=1$ we mean $y_d=y$, $x_0=r$ and $x_d=x$).
We define
$$R(y)=V(y_d)-g(y), \quad y\in \Rd.$$
We note that $R\ge 0$ and $R(x)=0$.
Since $V(y_d)$ is harmonic for $X_t$ at $y_d>0$, we have
\begin{eqnarray*}
-\!{\cal A}_t g(x)&=&
{\cal A}_t R(x)= \frac1{\E^x\tau_{B(x,t)}}\E^x[R(X_{\tau_{B(x,t)}})]\ge 0.
\end{eqnarray*}
In fact, by (\ref{Ikeda-Watanabe2}),
\begin{eqnarray}
{\cal A}_t R(x)&=&\frac1{\E^x\tau_{B(x,t)}}{
\E^x [R(X_{\tau_{B(x,t)}}), X_{\tau_{B(x,t)}}\in B(x,2t)]}\nonumber\\
&&+\frac1{\E^x\tau_{B(x,t)}}\int_{B(x,2t)^c}
R(y)
\int_{B(x,t)}\nu(y-w)G_{B(x,t)}(x,w)dw
dy.\label{geR}
\end{eqnarray}
We shall split the integral
into several parts.
First, if $y\in B_{r/2}^c\subset B(x,t)^c$ and $w\in B(x,t)$, then $\nu(y-w)\le \nu(3y/8)\le \nu(y/4)$, and by \eqref{defsD},
$$
\int_{B(x,t)}\nu(y-w)G_{B(x,t)}(x,w)dw\le \E^x\tau_{B(x,t)}\nu(y/4).
$$
By this, change of variables, \eqref{subad}, integration by parts and \eqref{tails},
\begin{eqnarray*}
&&\frac1{\E^x\tau_{B(x,t)}}\int_{B_{r/2}^c}
R(y)
\int_{B(x,t)}\nu(y-w)G_{B(x,t)}(x,w)dw
dy\le
\int_{B_{r/2}^c}R(y)\nu(y/4)dy\\
&&\le
\omega_d\int_{r/2}^\infty\nu(\rho/4)\rho^{d-1}V(\rho)d\rho
\le 4^{d+1}L(r/8)V(r/8)+4^{d+1}\int_{r/8}^\infty V'(\rho)L(\rho)d\rho\\
&&\le
c/V(r/8)+c\int_{r/8}^\infty V'(\rho)/V^2(\rho)d\rho
\le c/V(r).
\end{eqnarray*}
If $d=1$, then $R(y)=0$ on  $B_{r/2}=(-r/2,r/2)$ and the proof is complete.

In what follows we assume that $d\geq 2$ and \eqref{HR} holds.
We
denote (half-ball)
$F= B(x_0/2,r/2)\cap\{y_d<r/2\}=\{y: |y|^2/r<y_d<r/2\}$,
and we have
\begin{equation}\label{odod}
y_d/2\le \delta_{B(x_0,r)}(y)\le y_d \quad \text{ and } \quad y_d-\delta_{B(x_0,r)}(y)\le |\tilde{y}|^2/{r},
\quad \text{ if } y\in F
\end{equation}
(see the right side of Figure~1). We leave
verification of \eqref{odod} to the reader.
By \eqref{HR} and \eqref{odod},
\begin{equation}\label{pVwF}
R(y)\le H_rV'(y_d/2) \frac {|\tilde{y}|^2}{r}, \quad y\in F.
\end{equation}
If
$y\in B(x,2t)\subset F$, then by \eqref{x0} and  \eqref{odod} we further have
\begin{equation}\label{bRl}
R(y)\le 4H_r  V'(x_d/4) \frac {t^2}{r}.
\end{equation}
By  \eqref{bRl}  and Lemma \ref{LimitV(t)/t},
\begin{eqnarray*}
&&
\frac1{\E^x\tau_{B(x,t)}}
\E^x [R(X_{\tau_{B(x,t)}}), X_{\tau_{B(x,t)}}\in B(x,2t)]\\
&\le& \frac{4H_r}{\E^x\tau_{B(x,t)}}\;   V'(x_d/4) \frac {t^2}{r}
\le c H_rV'(x_d/4) \frac {t^2}{rV^2(t)}\to {c H_rV'(x_d/4) \frac {\sigma^2}{r}}, \quad \text{as $t\to 0$.}
\end{eqnarray*}
If $\sigma>0$, then by Lemma \ref{kappaDrift} we have $\sup_{x>0}V'(x)\le 1/\sigma$,
hence $V(r)\le r/\sigma$ and so
$$V'(x_d/4) \frac {\sigma^2}{r}\leq \frac{\sigma}{r}\leq \frac{1}{V(r)}.$$
If $y\in B_{r/2}\setminus B(x,2t)^c$ and $w\in B(x,t)$, then $|y-w|\ge |y-x|/2$.
Thus, \eqref{ooD} follows if
\begin{equation}\label{gen}
\int_{B_{r/2}}R(y)\nu\(\frac{y-x}{2}\)dy\le \frac {C_5H_r}{V(r)}. \end{equation}
To prove \eqref{gen},
we note the  singularity at $y=x\in F$,
cover $B_{r/2}$ with sets  $\{y\in F:|y_d-x_d|\le |\tilde{y}|\}$,
$\{y\in F:|\tilde{y}|<|y_d-x_d| \}$, $\{y\in \Rd: |\tilde y|<r/2, -r/2<y_d\le |y|^2/r \}$, and consider the corresponding integrals.
By \eqref{pVwF},
and Lemma {\ref{VIntegralEstimate}}, the first integral does not exceed
\begin{eqnarray*}
\frac {H_r}{r}\omega_{d{-1}}\int_0^{r}V'(y_d/2)\int_{ |y_d-x_d|}^r  \rho^d\nu(\rho/2)d \rho \le \frac{CH_r}{V(r)}.
 \end{eqnarray*}
Similarly, using \eqref{pVwF} and Lemma {\ref{VIntegralEstimate}}, we bound the second integral by
\begin{align*}
&\frac {H_r}{r}\int_0^{r} V'(y_d/2)\nu\(|y_d-x_d|/2\)\int_{  |\tilde{y}| < |y_d-x_d|}  |\tilde{y}|^2 d\tilde{y}dy_d\\
&=\frac {CH_r}r\int_0^{r} V'(y_d/2)\nu(|y_d-x_d|/2)|y_d-x_d|^{d+1} dy_d \le \frac{CH_r}{V(r)}.
\end{align*}
By a change of variables, \eqref{subad}, and Proposition~\ref{Vintegral},
we bound the third integral by
\begin{eqnarray}
&& \int_0^{r/2}V(s)  \int\limits_{rs-s^2\le |\tilde{y}|^2<(r/2)^2}\nu(\tilde{y}/2)d\tilde{y}ds\\
&&
=\int_0^{r/2}\rho^{d-2}\nu(\rho/2)d\rho \int_0^{{2}\rho^2/r} V(s) ds
  \le
\frac2r\int_0^{r/2}\rho^{d}\nu(\rho/2)V({2}\rho) d\rho\le \frac C {V(r)}. \end{eqnarray}
This completes  the proof of (\ref{gen}), and so the proof of the lemma.
\end{proof}

\begin{lem}\label{Vestimate2}
Assume that $\A $ holds or $d=1$.
Let $x_0\in \Rd$,
$r>0$ and  $g(x) = V(\delta_{B^c(x_0,r)}(x))
$. There is a constant $C_6=C_6(d)$
such that
\begin{equation}\label{subV}
0\le \limsup_{t\to 0}{\cal{A}}_t g(x)\le  \frac {C_6H_r}{V(r)},\qquad \text{ if }\quad 0<\delta_{B^c(x_0,r)}(x){<} r/4.
\end{equation}
\end{lem}
\begin{proof}
As in the proof of Lemma~\ref{Vestimate1}, we use the notation $y=(\tilde{y}, y_d)$ and without loosing generality we consider $x=(0, x_d)$,  $0<4t<x_d<r/4$, and  $x_0=(0,-r)$ (in dimension $d=1$ we mean $y_d=y$, $x_0=r$ and $x_d=x$).
This time we define
$$R(y)=g(y)-V(y_d), \quad y\in \Rd.$$
We have $R\ge 0$ and $R(x)=0$.
Since $V(y_d)$ is harmonic for $X_t$ at $y_d>0$,
\begin{eqnarray*}
{\cal A}_t g(x)&=&
{\cal A}_t R(x)= \frac1{\E^x\tau_{B(x,t)}}\E^x[R(X_{\tau_{B(x,t)}})]\ge 0.
\end{eqnarray*}
To prove \eqref{subV} we repeat verbatim the proof of Lemma \ref{Vestimate1}, starting from \eqref{geR} there, except for the following minor
modification: we replace \eqref{odod} with
\begin{equation}\label{odod2}
y_d\le \delta_{B^c(x_0,r)}(y)\le 3y_d/2 \quad \text{ and } \quad \delta_{B^c(x_0,r)}(y)-y_d\le |\tilde{y}|^2/{r},
\quad \text{ if } y\in F.
\end{equation}
\end{proof}

\section{Estimates of the expected exit time}\label{secExit}
Unless explicitly stated otherwise, we keep assuming that $X$ is a unimodal L\'{e}vy process in $\Rd$ with unbounded L\'evy-Kchintchine exponent $\psi$.
The following theorem gives a sharp estimate
for the expected exit time of the ball. Recall that the upper bound in Theorem~\ref{Exit2} actually holds for arbitrary rotation invariant L\'{e}vy process, as proved in Lemma \ref{ExitTimeUpper}.
\begin{thm}\label{Exit2}If
$\A $ holds, then there is
$C_7=C_7(d)$
such that for $r>0$,
\begin{equation}\label{eq:eet}\frac{C_7}{H_r} V(\delta_{B_r}(x))V(r)\le \E^x\tau_{B_r}\le 2 V(\delta_{B_r}(x))V(r), \quad x\in \Rd.
\end{equation}
\end{thm}
\begin{proof}
Due to Lemma \ref{ExitTimeUpper} it 
suffices to prove 
the lower bound in \eqref{eq:eet}.
Of course it
holds on $\overline{B}^c_r$.
Denote $s(x)=\E^x \tau_{B_r}$ and $g(x)=V(\delta_{B_r}(x))$, $x\in \Rd$.
By (\ref{eq:2a}), domain-monotonicity of the exit times  and \eqref{subad}, the bound holds on $\overline{B_{r/4}}$, i.e.
there is $C=C(d)$ so large that
$$Cs(x)- V(r)g(x)\ge 0 \qquad \text{ if } \quad \delta_{B_r}(x)\ge r/4.$$
Define $0<\delta_{B_r}(x)< r/4$.
If $t>0$ is small, then by Lemma \ref{Vestimate1} we have
$|{\cal{A}}_t g(x)|\le C_5H_r/V(r)$, and by \eqref{sx-1} we obtain
$${\cal{A}}_t\left[ (C_5H_r+1) s- V(r)g\right](x) =  -(C_5H_r+1)- V(r) {\cal{A}}_t g(x)\le -1.
$$
Let $c=C\vee (C_5H_r+1)$ and $f(x)= c s(x)- V(r)g(x)$, a continuous function.
By Corollary~\ref{max_p}, $f$ cannot attain global minimum on $B_r\setminus \overline{B_{3r/4}}$. Since $f\ge 0$ elsewhere, $f\ge 0$ everywhere.
\end{proof}

The above argument was inspired by the proof of Green function estimates for the ball and stable L\'evy processes given by K.~Bogdan and P.~Sztonyk in \cite{MR2320691}.

\begin{cor}\label{csbeD}
If $D$ is bounded, convex and $C^{1,1}$ at scale $r>0$, and
if $\A $ holds, then
\begin{equation}
\frac{C_7}{H_r} V(\delta_{D}(x))V(r)\le \E^x\tau_{D}\le  V(\delta_{D}(x))V(\diam( D)), \quad x\in \Rd.
\end{equation}
\end{cor}
\begin{proof}
Consider strip $\Pi\supset D$ of width not exceeding $\diam(D)$ and ball $B\subset D$ of radius $r\vee \delta_D(x)$ such that $\delta_D(x)= \delta_\Pi(x)=\delta_B(x)$. Since $s_B(x)\le s_D(x)\le s_\Pi(x)$,
the result follows from
\eqref{exitTimeOneDim} and Theorem~\ref{Exit2}.
\end{proof}

\begin{remark}
All the results
in this section also hold if
$\nu$ is isotropic, infinite and {\it approximately 
unimodal} in the sense of  \eqref{generalize1} below.
Here is an example and explanation.
\end{remark}

\begin{cor}\label{genaral1}
Let $X$
be isotropic with  absolutely continuous L\'{e}vy measure $\nu(dx)=\nu(|x|)dx$.
Let $\nu_0:(0,\infty)\to(0,\infty)$ be monotone and
let $C^*$ be a constant such that
\begin{equation}\label{generalize1}(C^*)^{-1}\nu_0(r)\leq \nu(r)\leq C^*\nu_0(r), \quad r>0.\end{equation}If  
$\A $ holds, then there is
$c=c(d,C^*)$
such that for $r>0$,
$$ \E^x\tau_{B_r}\geq \frac{c}{H_r} V(\delta_{B_r}(x))V(r), \quad x\in \Rd.$$
\end{cor}
\begin{proof}
{Let $Y$ be unimodal with
characteristic function $\psi^Y(\xi)=\sigma|x|^2+\int_{\Rd}(1-\cos\langle \xi, z\rangle)\nu_0(|z|)dz$.
By Proposition \ref{ch1Vp} we
have $V^Y(r)\approx V(r)$, $r>0$.
By Proposition \ref{Vintegral},
\begin{equation}\label{general1-1}\int^r_0V(\rho)\rho^d\nu(\rho)d\rho\approx \int^r_0V^Y(\rho)\rho^d\nu_0(\rho)d\rho\leq C\frac{r}{V(r)},\quad r>0.\end{equation}
The
inequality and approximate monotonicity of $\nu$
yield extensions
of Lemmas \ref{VIntegralEstimate} and \ref{Vestimate1}, from which the present
corollary
follows in a similar manner as Theorem \ref{Exit2}. }
\end{proof}

For $r>0$ we define non-increasing functions
\begin{equation}\label{defI}
\mathcal{I}(r)=\inf_{ 0<\rho\leq r/2}\left[\nu(B_{r}\setminus B_\rho) V^2(\rho)\right],
\end{equation}
and
\begin{equation}\label{defJ}
   \mathcal{J}(r)=\inf_{ 0<\rho\leq r}\left[L(\rho) V^2(\rho)\right].
\end{equation}
By
\eqref{tails}, $0\le \mathcal{I}(2r
)\leq\mathcal{J}(r)\leq c(d)$.
We shall use $\mathcal{J}$ immediately,
but
$\mathcal{I}$ shall only be  discussed and used
later, in Sections~\ref{sec:ScalingAndConsequences} and~\ref{Survival}.
\begin{lem}\label{exit_time}Let
$\A $ hold.   Denote $D=B_{1}^c$.
Let $0<r<1$, $x\in D$ and $0<\delta_D(x)\le r/2$. Let  $x_0=x/|x|$ and
$D_1= B(x_0, r)\cap D$.   There is
$C_{8}=C_8(d)$  such that
\begin{equation}\label{seas}
\E^x \tau_{D_1}\le C_{8}\frac{H_1 }{\(\mathcal{J}(1)\)^{2}}  V(\delta_D(x))\,V(r)\,.
\end{equation}
\end{lem}
\begin{proof}
If $a,b,c\ge 0$, $k\ge 2$, $a-b+c\ge 0$ and $b\ge kc$, then  $a\ge b-c\ge (k-1)c\ge kc/2$, or $c\le 2a/k$. We shall use this observation to compare
$a({v}) = {V(\delta_{D}(v))}$, $b(v)= \E^v a(X_{\tau_{D_1}})$ and
$s(v)= \E^v \tau_{D_1}$, where $v\in \Rd$. We first let $0<r{\leq}1/4$, and
consider
	$$f(v)= a(v) -b(v)+\frac {{C_6H_1+1}} {V(1)} s(v),\quad v\in \Rd.$$
If $v\notin D_1$ then $a(v) = b(v)$, $s(v)=0$, and so $f(v)= 0$.
By Lemmas~\ref{lcont} and \ref{lem:cfh}, $f\in C_0(D_1)$.
If $v\in D_1$, and $t>0$ is small enough, then Lemma \ref{Vestimate2} and \eqref{sx-1} yield
$${\cal{A}}_t f(v) \le  \frac {{C_6H_1}}{V(1)}-\frac {{C_6H_1+1}}{V(1)}<0.
$$
By Lemma \ref{lc0}, $f\ge 0$ on $D_1$.
Let $F=\{y: y_d>1+r\}$. We note that by Ikeda-Watanabe,
\begin{eqnarray*}b(v)&\ge& V(r) P^v( X_{\tau_{D_1}}\in F)\ge V(r) E^v \tau_{D_1}  \inf_{z\in D_1}\nu(F-z)\\
&=&  V(r) E^v \tau_{D_1} L_1(r+r^2/2)\ge V(r) E^v \tau_{D_1} L_1(9r/8).
\end{eqnarray*}
Since $X_t$ is rotation invariant, there is $c_1=c_1(d)$ such that
 $L_1(9r/8)\geq 4c_1L(2r)\ge \frac{4c_1\mathcal{J}(1)} {V^2(2r)}
 \ge \frac{c_1\mathcal{J}(1)} {V^2(r)}$, $r\leq 1/4$. Hence,
$$ b(v)\ge    E^v \tau_{D_1}  \frac{c_1\mathcal{J}(1)} {V(r)}\geq\frac{c_1\mathcal{J}(1)}{C_6H_1+1}\frac{V(1)}{V(r)}\frac {C_6H_{1}+1} {V(1)} s(v).$$
If $\frac{c_1\mathcal{J}(1)}{C_6H_1+1}\frac{V(1)}{V({1/4})}\ge 2$,
then we let $r_0=1/4$, else
 we pick $r_0>0$ so that  $\frac{c_1\mathcal{J}(1)}{C_6H_1+1}\frac{V(1)}{V({r_0})}= 2$.
By the observation at the beginning of the proof, for $0<r\le r_0$,
we have $s(v)\le 2V(\delta_D(v))V(r)/(c_1\mathcal{J}(1))$ for all $v$, in particular for $v=x$.
	
	For $r_0<r<1$ we proceed in the following standard way. First assume that $\delta(x)\le r_0/2$, and let $D^\prime= B(x_0, r_0)\cap D$.  Then by the strong Markov property,
		$$s(x)=\E^x \tau_{D_1}=\E^x \tau_{D^\prime}+\E^xs(X_{\tau_{D^\prime}}).$$
	As stated in Theorem~\ref{Exit2}, $s(x)\le 2V^2(r)$. By Lemma~\ref{upper1}, we thus obtain,
	$$\E^x s(X_{\tau_{D^\prime}})\le 2V^2(r)\p^x(|X_{\tau_{D^\prime}}-x_0|\ge r_0)\le 2C_1 V^2(r)\frac {\E^x \tau_{D^\prime}}{V^2(r_0)}.$$
	 If this is combined with the estimates already proved, then  $c_2=c_2(d)$ exists such that
	$$
s(x)\le (2C_1+1) \E^x \tau_{D^\prime}  \frac {V^2(r)}{V^2(r_0)}\le  c_2  \frac {V(r)}{\mathcal{J}(1)V(r_0)} V(\delta_D(x))V(r)\le c_2  \frac {V(1)}{\mathcal{J}(1)V(r_0)} V(\delta_D(x))V(r). $$
If $\delta_D(x)\ge r_0/2$, then 
by Lemma \ref{ExitTimeUpper}  and subaddativity of $V$, we trivially have
$$
s(x)\le 2V^2(r)\le \frac {2V(r)}{V(r_0/2)}V(\delta_D(x))V(r)\le \frac {4V(1)}{V(r_0)}V(\delta_D(x))V(r) .$$
Summarizing, by taking $c_3=4+ c_2$,  in all the cases we get
$$\E^x \tau_{D_1}\le c_3 \frac {V(1)}{V(r_0)} \left( 1+ \frac {1}{\mathcal{J}(1)}\right) V(\delta_D(x))V(r).$$
By the choice of $r_0$  and \eqref{tails},  
$V(1)/V(r_0){\leq   4+ (C_6H_1+1)/(c_1\mathcal{J}(1))\leq  c_4H_1/\mathcal{J}(1)}$, where $c_4=c_4(d)$. Therefore,
$$\E^x \tau_{D_1}\le 
c_4H_1  \frac {\mathcal{J}(1)+1}{\mathcal{J}(1)^2 } V(\delta_D(x))V(r).$$
This is equivalent to \eqref{seas}.
\end{proof}
\begin{cor}\label{exit_time_R}
Let
$\A $ hold.   Denote $D=B_{R}^c$.  Let $0<r<R$, $x\in D$, $0<\delta_D(x)\le r/2$,  $x_0=xR/|x|$. If
$D_1= B(x_0, r)\cap D$, then
\begin{equation}\label{seas1}
\E^x \tau_{D_1}\le C_{8}\frac{H_R}{\(\mathcal{J}(R)\)^{2}}  V(\delta_D(x))\,V(r)\,.
\end{equation}
\end{cor}
\begin{proof}
Let $R>0$, $Y_t=X_{ t}/R$ and denote by $V_Y,\, \tau^Y_{B},\, L_Y,\,
\mathcal{J}^Y_\infty$ the
quantities $V$, $\tau_B$, $L$, $\mathcal{J}$ corresponding to $Y$.
By (\ref{kappa}), we infer that  $V_Y(s)= V(Rs),\, s\geq 0$.
Furthermore,  $L_Y(s)= L(Rs)$ for $s>0$. Hence,  we obtain $V_Y(s)L_Y(s )=  V^2(Rs)L(Rs),$  which shows that $\mathcal{J}^Y(1) =\mathcal{J}(R)$. Furthermore, $A^Y_1=H_R$.
We also have
$$\E^x\tau_{D_1}=\E^{x/R}\tau^Y_{D_1/R}.$$
Here the expectation on the right hand side corresponds to $Y$.
Lemma \ref{exit_time} finishes the proof.
\end{proof}

\noindent
The above argument shall be called {\it scaling}. (A different, {\it weak} scaling is discussed in Section~\ref{sec:ScalingAndConsequences}.)

\begin{remark}
If $\mathcal{J}(R)=0$ but $\mathcal{J}(R_1)>0$ for some $R_1<R$, then we
may replace
$(\mathcal{J}(R))^2$  in \eqref{seas1} by $V(R_1)\mathcal{J}(R_1)/V(R)$.
This follows from the proof of the Lemma \ref{exit_time}.
\end{remark}

The following is one of our main results.

\begin{thm}\label{Exit_C11}
If
$\A $ holds and $D\subset \Rd$ is open, bounded and $C^{1,1}$ at scale $r>0$, then
$C_9=C_9(d)$ and $C_{10}=C_{10}(d)$ exist such that
$$\frac{C_9}{H_r}\, {V(\delta_{D}(x))V(r)}\le    \E^x\tau_{D}\le  C_{10}\frac{H_r}{\(\mathcal{J}(r)\)^{2}} \, \frac{V^2(\diam D)}{V^2(r)}\, {V(\delta_{D}(x))V(r)}, \quad x\in \Rd.$$
\end{thm}
\begin{proof}
{Denote $s(x)=E^x\tau_D$.
By Lemma \ref{ExitTimeUpper}, $s(x)\leq V^2(\diam( D))$. Let $Q\in \partial D$ be such that $|x-Q|=\delta_D(x)$.
Let $\delta_D(x)\leq r/2$.  Since $D$ is $C^{1,1}$ at scale $r$, there exist $x_1\in D^c$ and $x_2\in D$ such that $B(x_1,r)\subset D^c$, $B(x_2,r)\subset D$
and $\{Q\}= \overline{B(x_1,r)}\cap\overline{B(x_2,r)}$.
  Let  $D_1=B(Q,r)\cap D$. By the strong Markov property and \eqref{eq:l},
\begin{eqnarray*}s(x)&=&E^x\tau_{D_1}+E^xs(X_{\tau_{D_1}})\leq E^x\tau_{D_1}+V^2(\diam(D))P^x(|X_{\tau_{D_1}}-Q|>r)\\
&\leq& E^x\tau_{D_1}\(1+C_1\frac{V^2(\diam(D))}{V^2(r)}\).
\end{eqnarray*}
Corollary \ref{exit_time_R} yields
the upper bound, since $E^x\tau_{D_1}\le E^x\tau_{D_2}$,    where $D_2= B(Q,r)\cap \overline{B(x_1,r)}^c$. The lower bound is a consequence of Theorem \ref{Exit2}, because $s(x)\ge E^x\tau_{{B(x_2,r)}}$.}

For the case $ \delta_D(x)\geq r/2$, we see from \eqref{eq:2a} that $s(x)\ge E^x\tau_{{B(x,\delta_D(x))}}\ge C_1^{-1}V^2(\delta_D(x))\ge (2C_1)^{-1}V(\delta_D(x))V(r)$. By this,
the general upper bound $s(x)\leq 2V^2(\diam( D))$ and the observations that $H_r\ge 1$ and $\mathcal{J}(r)\le c(d)$, we finish the proof.
\end{proof}
Note that $V^2(\diam D)/V^2(r)$ is bounded by the square of the distortion of $D$, if $r=r_0$.

In the one-dimensional case in the proof of Theorem~\ref{Exit_C11} we may apply \eqref{exitTimeOneDim} instead  of Theorem \ref{Exit2}  and Corollary \ref{exit_time_R}, to
obtain the following improvement.
\begin{cor}
If $X$
is a symmetric L\'{e}vy process in $\R$ with unbounded L\'{e}vy-Kchintchine exponent, and $D\subset \R$ is open, bounded and $C^{1,1}$  at scale $r>0$,
then
absolute constant $c\ge1$ exists such that
$c^{-1}{V(\delta_{D}(x))V(r)} \leq { \E^x\tau_{D}}\le  c V^2(\diam D)V^{-2}(r)\,{V(\delta_{D}(x))V(r)}$ for $x\in \R$.
\end{cor}

\section{Scaling and its consequences  }\label{sec:ScalingAndConsequences}

Let $X$ be an isotropic {unimodal} L\'{e}vy process in $\Rd$ with 
infinite L\'evy measure $\nu$.

In view of the
literature
of the subject (cf. \cite{2013arXiv1305.0976B}),
power-like asymptotics  of the characteristic exponent of $X$ is a natural condition to consider.
Let $I=(\lt,\infty)$, where $\lt\in [0,\infty)$ and let
 $\phi\ge 0$ be a
non-zero  function on $(0,\infty)$.
We say that
$\phi$ satisfies {the} {\it weak lower scaling} condition (at infinity) if there are numbers
$\la>0$
and  $\lC
{\in(0,1]}$,  such that
\begin{equation}\label{eq:LSC}
 \phi(\lambda\theta)\ge
\lC\lambda^{\,\la} \phi(\theta)\quad \mbox{for}\quad \lambda\ge 1, \quad\theta
{\in I}.
\end{equation}
In short we say that $\phi$ satisfies WLSC($\la, \lt,{\lC}$) and write $\phi\in\WLSC{\la}{ \lt}{\lC}$.
If $\phi\in\WLSC{\la}{0}{\lC}$, then we say
that $\phi$ satisfies the {\it global} weak lower scaling condition.

Similarly, we consider
$I=(\ut,\infty)$, where $\ut\in [0,\infty)$ and we say that
the weak upper scaling condition holds if there are numbers $\ua{<2}$
and $\uC{\in [1,\infty)}$ such that
\begin{equation}\label{eq:USC}
 \phi(\lambda\theta)\le
\uC\lambda^{\,\ua} \phi(\theta)\quad \mbox{for}\quad \lambda\ge 1, \quad\theta
{\in I}.
\end{equation}
In short, $\phi\in\WUSC{\ua}{ \ut}{\uC}$. For {\it global} weak upper scaling we require $\ut=0$ in \eqref{eq:USC}.
We write $\phi\in$ WLSC or WUSC if the actual values of the parameters are not important.
We shall study consequences of WUSC and WLSC for  the characteristic exponent $\psi$ of $X_t$.

Recall that $\psi$ is a radial function
and we use the notation $\psi(u)=\psi(x)$, where $x\in \Rd$ and $u=|x|$.
Our estimates below are expressed in terms of  $V$, $\psi$ or $\psi^*$.
In view of Proposition~\eqref{ch1Vp}, these functions yield equivalent descriptions ($\psi$ or $\psi^*$ are even comparable, see \eqref{eqcpg}).
Our main goal is to find
connections between the scaling conditions on $\psi$
and the magnitude of the quantities $\mathcal{J}$ and $\mathcal{I}$ defined in \eqref{defJ} and \eqref{defI}.
In the
preceding section we saw that $\mathcal{J}$ plays a
role in estimating the expected exit  time from $C^{1,1}$ open sets.
The next three results prepare analysis of survival probabilities  in Section \ref{Survival}. The first one comes from \cite[Corollary 15]{2013arXiv1305.0976B}.
\begin{lem}\label{GApprox}
$C=C(d)$ exists such that if $\psi{\in}$
WUSC$(\ua,\ut,\uC)$, $a=[(2-\ua)C]^{\frac2{2-\ua}}\uC^{\frac{\ua-2}{2}}$,  then
$$L(r)\geq a \psi(r^{-1}), \qquad {0<}r\leq \sqrt{a}/\ut.$$
\end{lem}

\begin{prop}\label{inuV}
\begin{description}
  \item[(i)] $\psi$ satisfies WUSC if and only if there is
$R
>0
$, such that
$\mathcal{J}(R )>0$.
  \item[(ii)] $\psi$ satisfies WUSC and WLSC (global WUSC and WLSC) if and only if
for some $R >0
$ ($R=\infty$, resp.) we have $\inf_{
{r}
<  R }\mathcal{I}(r)>0$.
\end{description}
\end{prop}
\begin{proof}
Assume that $\psi$ satisfies WUSC$(\beta_1,\theta,\uC)$. 
By Lemma \ref{GApprox} and \eqref{cVh1pgstare}, there  is a constant $c_1$ such that $L(r)V^2(r)\geq c_1>0$ for $r\leq \sqrt{a}/\theta$, and so $\mathcal{J}(r)\geq c_1>0$
for such $r$.
On the other hand, if
$\mathcal{J}(R )
>0$, then $L(r)\geq
\mathcal{J}(R )/V^2(r)$  for $r\leq R $ . By the proof of \cite[Theorem 26]{2013arXiv1305.0976B} there is a complete Bernstein function $\phi$ with the L\'{e}vy density $\nu$ such that \begin{equation}\label{NWSR2}f(r)=\int^\infty_r\nu(u)du\geq c \int_r^\infty u^{(d-2)/2}\nu(u^{1/2})du
=cL(r^{1/2}), \quad r>0.\end{equation}
Furthermore, $\mathcal{L}f(r)=r/\phi(r)$, $r>0$.
 By (\ref{NWSR2}),  (\ref{cVh1pgstare}) and \cite[Proposition 2]{2013arXiv1305.0976B}, 
\begin{align*}
f(r)&\ge c/V^2(r^{1/2})\ge c \psi(r^{-1/2}).
\end{align*}
Hence, arguments based on \cite[(27) and (32)]{2013arXiv1305.0976B} as in
the proof of  \cite[Theorem 26]{2013arXiv1305.0976B}, yield WUSC$(\beta_1,R ^{-1},\uC)$ for $\psi$ for some $\beta_1<2$ and $\uC\geq 1$.

To prove the second part of the statement we
suppose that $\psi$ satisfies WUSC$(\beta_1,\theta,\uC)$ and WLSC$(\beta_2,\theta,\lC)$.
By \cite[Corollary 22]{2013arXiv1305.0976B} and (\ref{cVh1pgstare}),
$$\nu(x)\ge \frac {c^*}{ V^2(|x|)|x|^d},\quad |x|\leq b/\theta.$$ If $2
\rho\leq b/\theta$, then by monotonicity of $V$ we have
$$V^2(
\rho)\nu(B_{2\rho}\setminus B_{\rho})\ge V^2(\rho)\int_{B_{2\rho}\setminus B_{\rho}} \frac {c^*dx}{ V^2(|x|)|x|^d}\ge \int_{B_{\rho}\setminus B_{\rho/2}} c^*\frac {dx}{ |x|^d}.
$$
Therefore $\mathcal{I}(r)\ge c^*(1-(1/2)^d)\omega_d/d$ for all $r\le b/\theta$, as needed.

To prove the reverse implication we assume that there exist constants $c^*$ and $R $, such that for $0<
r< R $,
$\mathcal{I}(
r)\geq c^*$. By radial monotonicity of $\nu$,
$$\nu(x)\geq \frac{c^*}{|B_
2|-|B_{1}|}\frac{1}{V^2(|x|)|x|^d}, \quad |x|< R /2.$$
By (\ref{cVh1pgstare})
and  \cite[Theorem 26]{2013arXiv1305.0976B}, we obtain WLSC and WUSC for $\psi$.
\end{proof}

\begin{prop}\label{rem_inuV}
If $\psi$ satisfies  WUSC but not WLSC, then $\liminf_{r\to 0}\mathcal{I}(r)=0$
but there is ${R}
>0$ such that  $\mathcal{I}({r}
)>0$ for $r< R$.
\end{prop}
\begin{proof}
Let $R=2\sup\{r: \nu(r)>0\}$. We have $R>0$.
If $\psi$ satisfies WUSC, then by Lemma \ref{GApprox} and Proposition \ref{ch1Vp}, there are $c_1,r_1>0$, such that  $L(\rho)\geq c_1/V^2(\rho)$ for $\rho< r_1$.
Since $\lim_{\rho\to 0}V(\rho)=0$,
$$\liminf_{\rho\to0}V^2(\rho)\nu(B_r\setminus B_\rho)=\liminf_{\rho\to0}V^2(\rho)L(\rho)\geq c_1,$$
for every $r>0$. Fix $r\in (0,R)$. There is $r_2>0$ such that $$V^2(\rho)\nu(B_r\setminus B_\rho)\geq c_1/2  \quad \text{ if } \rho\leq r_2.$$
If $r_2< \rho\leq r/2$, then by monotonicity  of $V$, $$V^2(\rho)\nu(B_r\setminus B_\rho)\geq V^2(r_2)\nu(B_r\setminus B_{r/2})>0,$$
hence $\mathcal{I}(r)>0$. If $\liminf_{r\to 0}\mathcal{I}(r)>0$, then by Proposition \ref{inuV}, $\psi$ satisfies also WLSC.
\end{proof}

\subsection{Hitting a ball}\label{sec:hb}
We shall estimate the probability that $X$ ever hits a fixed ball of radius $R>0$.
If $X$ is transient and its starting point is far from the ball, then the probability of such an event is small; $X$ instead drifts to infinity with probability bounded below by a positive constant.
Indeed, define
$$U(x)=\Z p_t(x)dt,\qquad x\in \Rd,$$
the potential kernel of $X$.
If the process is transient  \cite{MR0346919}, then $U$ is finite almost everywhere, in fact on $\Rdz$.
This is the case, e.g. if $d\geq 3$.
We denote by $\CAP$ the capacity with respect to $X$. Recall that for every non-empty compact set $A\subset \Rd$ there exists a measure $\mu_A$, supported on $A$ (see, e.g., \cite[Section II.2]{MR1406564}), called the equilibrium measure, such that
\begin{equation}\label{CapDef}U\mu_A(x)=\int U(x-y)\mu_A(dy)=\p^x(\tau_{A^c}<\infty),\quad x\in \Rd,\end{equation}
and $\mu_A(A)=\CAP(A)$.
The following two lemmas were proved in \cite{2013arXiv1301.2441G}.
\begin{lem}\label{pot} If
$d\ge 3$, then there is $C_{15}=C_{15}(d)$ such that
$$ U(x)\le C_{15}\frac {V^2(|x|)}{|x|^d}, \qquad x\in \R^d.$$
\end{lem}
We note in passing that lower bounds for $U$ are given in \cite{2013arXiv1301.2441G} under WLSC.
\begin{lem} \label{cap}If
$d\ge 3$, then there is $C_{16}=C_{16}(d)$ such that
$$C_{16}^{-1}\frac {
R^d} {V^2(R)} \le \CAP(\overline{B_R})\le C_{16}\frac {
R^d} {V^2(R)}, \qquad R>0.$$
\end{lem}
If
$\psi\in$ 
WUSC($
\ua,0,\uC$) and $d>\ua
>0$, then the process $X$ is transient (even if $d<3$),
and we may
extend
the two previous lemmas
by using the weak upper scaling condition instead of \cite[Lemma 3]{2013arXiv1301.2441G} (see the last part of Section 4 in \cite{2013arXiv1301.2441G} for more details). Here are the resulting statements.
\begin{lem}\label{pot1} If $\psi\in\WUSC{\ua}{0}{\uC}$ and $\ua<d\leq 2$,
then
$c=c(d,\ua,\uC)$ exists such that
$$ U(x)\le c\frac {V^2(|x|)}{|x|^d}, \qquad x\in \R^d.$$
\end{lem}
\begin{lem} \label{cap1}
If $\psi\in\WUSC{\ua}{0}{\uC}$ and $\ua<d\leq 2$,
then
$c=c(d,\ua,\uC)$ exists such that
$$c^{-1}\frac {r^d} {V^2(
r)} \le \CAP(\overline{B_r})\le c\frac {r^d} {V^2(r)}, \quad r>0. $$
\end{lem}
As a consequence of the above lemmas we obtain the  upper bound of the probability that the process ever hits a ball of arbitrary radius, a close analogue of the classical Brownian result.

\begin{prop}\label{CAP}For $d\geq 3$ there exists a constant $C_{17}=C_{17}(d)$ such that  for $|x|>R>0$,
\begin{equation}\label{eq:CAP}\p^x(\tau_{\overline{B}^c_R}<\infty)\leq C_{17} \frac{ V^2(|x|)}{|x|^d}:\frac{ V^2(R)}{R^d}.\end{equation}
{If  $\ua<d\leq 2$ and  $\psi\in$WUSC($\ua,0,\uC$)}, then  {\rm (\ref{eq:CAP})} holds   with $C_{17}=C_{17}(d,\ua,\uC)$.
\end{prop}
\begin{proof}  
We have
$$\p^x(\tau_{\overline{B}^c_R}<\infty)=\int_{\overline{B_R}} U(y-x)\mu_{B_R}(dy).$$
 By Lemma \ref{pot}, for $y\in B_R$ and $|x|\geq 2R$ we get
$$U(x-y)\leq  2^dC_{15}|x|^{-d}V^2(|x|).$$ Hence, by Lemma \ref{cap},
$$\p^x(\tau_{\overline{B}^c_R}<\infty)\le  2^dC_{15}|x|^{-d}V^2(|x|)\CAP(\overline{B_R})\leq 2^dC_{15}C_{16} \frac{R^{d}V^2(|x|)}{|x|^{d}V^2(R)}.$$
Since $[R^{d}V^2(|x|)]/[|x|^{d}V^2(R)]\geq 2^{-d}$, for $|x|\leq 2R$ we have
$$\p^x(\tau_{\overline{B}^c_R}<\infty)\leq 2^d(C_{15}C_{16}+1)  \frac{R^{d}V^2(|x|)}{|x|^{d}V^2(R)}, \quad |x|>R.$$
To prove the second claim we use Lemma \ref{pot1} and \ref{cap1} above instead of
\ref{pot} and \ref{cap}.
\end{proof}
The following result is important in Section~\ref{Survival}.
\begin{cor}\label{hit_infty} If $d\ge 3$, then
$c=c(d)$ exists
such that
$$\p^x(\tau_{\overline{B}^c_R}=\infty)\ge 1/2, \qquad |x|\ge c R.$$
{If $\ua<d\leq2$ and  $\psi\in$WUSC($\ua,0,\uC$), then the above inequality holds}
with $c=c(d,\ua,\uC)$.
\end {cor}

\section{Estimates of survival probability}\label{Survival}
In this section  we assume that $X$ is a pure-jump unimodal L\'{e}vy process with infinite L\'evy measure.

\begin{prop}\label{L5a}
Let
$\A $ hold. There are $ C_{11}=C_{11}(d)<1$ and $C_{12}=C_{12}(d)$ such that if $R> 0 $ and $t\le C_{11} V^2(R)$, then  
\begin{eqnarray*}\p^x(\tau_{B_{R}}>t)&\ge&   C_{12}\,\frac{\mathcal{I}(R)}{H_R}\left(\frac{V(\delta_{B_{R}}(x))}{\sqrt{t}}\wedge1\right).\end{eqnarray*}
\end{prop}

\begin{proof} Let $R=1$ and $C_{11}=C_3/64$.
 Due to Corollary~\ref{kula} and subadditivity of $V$,
\begin{equation}\label{kT1}{\p^0}(\tau_{B_{r/8}}>C_{11} V^2(r))\ge 1/2.\end{equation}
 Suppose that $0<t\le C_{11} V^2(1)$ and pick $r\le 1$ such that  $t = C_{11} V^2(r).$
Let $x\in B_{1}$. If $\delta_{B_{1}}(x)\ge {r/8},$ then $\p^x(\tau_{B_{1}}>t)\ge 1/2$ by \eqref{kT1}. To complete the proof {for $R=1$}, it is enough to consider the case $\delta_{B_{1}}(x)< {r/8}$. Let $\delta_{B_{1}}(x)<{r/8}$. Let {$r_0=r/2\wedge 1/4$} and $D_r=B_{1}\setminus  B_{1-r_0}$. {Notice that $B(z,r/4)\subset B_1$ for $z\in B_{1-r_0}$.}  By the strong Markov property,
\begin{eqnarray*}\p^x(\tau_{B_{1}}>t)&\ge&  \E^x\left[\p^{X_{\tau_{D_r}}}(\tau_{B_{1}}>t); X_{\tau_{D_r}}\in B_{1-r_0}\right]
\ge  \inf_{z\in B_{1-r_0}}\p^z(\tau_{B_{1}}>t) \p^x\left[ X_{\tau_{D_r}}\in  B_{1-r_0}\right] \\
&\ge&  \p^0(\tau_{B_{r/4}}>C_{11}V(r)) \p^x\left[ X_{\tau_{D_r}}\in B_{1-r_0}\right]
\ge  (1/2) \p^x\left[ X_{\tau_{D_r}}\in B_{1-r_0}\right].
\end{eqnarray*}
If $|z_0|={1}$, then by the Ikeda-Watanabe formula, isotropy and monotonicity of the L\'evy density,
$$\p^x\left[ X_{\tau_{D_r}}\in B_{1-r_0}\right]\ge \E^x \tau_{D_r}\inf_{z\in D_r}\nu(z-B_{1-r_0})\ge  \nu(z_0-B_{1-r_0})\E^x \tau_{D_r}.$$
 By Theorem  \ref{Exit_C11},  subaddativity of $V$,  
 $\E^x \tau_{D_r}\ge \frac{ C_9}{{H_{r_0/2}}} V(r_0/2)V(\delta_{B_{1}}(x))\ge \frac{ C_9}{{8H_{1}}} V(r)V(\delta_{B_{1}}(x))$. Since $\nu$ is isotropic,
$\nu(z_0-B_{1-r_0})\ge c_1\nu(B_{1}\setminus B_{2r_0} ) \ge   c_1\frac{\mathcal{I}(1)} {4V^2(r)} $, where
$c_1=c_1(d)$.
Therefore,
$$\p^x(\tau_{B_{1}}>t) \ge c_1\frac{ C_9}{64H_{1}}{\mathcal{I}(1)} \frac{V(\delta_{B_{1}}(x))}{V(r)}= C_{12} \frac{{\mathcal{I}(1)}}{H_1}  \frac{V(\delta_{B_{1}}(x))}{\sqrt{t}},$$
where $C_{12}=c_1C_9\sqrt{C_3}/512$.

For arbitrary $R>0$  we use scaling as in the proof of Corollary \ref{exit_time_R}.
\end{proof}
\begin{remark}\label{rem:spball_exp}The estimate in Proposition \ref{L5a} is sharp if $t\le C_{11}V^2(R)$;
a reverse inequality
follows immediately from Proposition \ref{lalfline}. If $t>C_{11}V^2(R)$, then one can use spectral theory to
observe exponential decay of the Dirichlet heat kernel and the survival probability in time
if, say, $\sup_x  p_t(x) <\infty$ for all $t>0$ (see \cite[Theorem 3.1]{MR2445505}, \cite[Corollary 7]{2013arXiv1305.0976B}, \cite[Theorem 4.2.5]{MR990239}).
\end{remark}

\begin{lem}\label{exit_ub_R}  Let $D=\overline{B}^c_R$
and  let
$\A $ hold. There is $C_{13}=C_{13}(d)$ such that,
 $$\p^x(\tau_{D}>t)\le  C_{13}{\frac{H_R}{ ( \mathcal{J}(R))^{2}}}\frac{V(\delta_D(x))}{\sqrt{t}\wedge V(R)}, \qquad t>0, \; x\in \Rd. $$
  \end{lem}

\begin{proof}
Let $x\in D$ and $x_0=xR/|x|$. If  $0<t\le V^2(R)$, then
we choose $r$ so that $V(r)= \sqrt{t}$, otherwise we set  $r=R$.
We define $$D_1= B(x_0, r)\cap B_{R}^c.$$
Since $ H_R\ge 1$ and $\mathcal{J}(R)\le c(d)$, we may assume  that $0<\delta_D(x)\le r/2$. By Corollary \ref{exit_time_R},
   $$\E^x \tau_{D_1}\le   C_{8} \frac{H_R}{(\mathcal{J}(R))^{2}} V(r)V(\delta_D(x)).$$
	By (\ref{eq:l}),
 $$\p^x(|X_{\tau_{D_1}}{-x_0}|\ge r)\le C_1 \frac{   \E^x \tau_{D_1}}{V^2(r)}.$$
Finally, we get the conclusion:
 \begin{eqnarray*}\p^x(\tau_D>t)&\le& \p^x(\tau_{D_1}>t)+ \p^x(|X_{\tau_{D_1}}{-x_0}|\ge r) \le \frac {\E^x \tau_{D_1}}t+ C_1\frac{   \E^x \tau_{D_1}}{V^2(r)}\\&\le&  (C_1+1) {C_{8}H_R}(\mathcal{J}(R))^{-2}\frac{V(\delta_D(x))}{\sqrt{t}\wedge V({R})}. \end{eqnarray*}
\end{proof}

\begin{remark} \label{1_dim_ball}If $d=1$, then regardless of
$\A $, we have for any $t>0$,
 $$\p^x(\tau_{D}>t)\le  C_{13}\frac{V(\delta_D(x))}{\sqrt{t}\wedge V(R)}. $$
 This is easily seen from the above proof and the estimate $\E^x \tau_{D_1}\le 2V(r/2)V(\delta_D(x))$. The estimate is not, however, sharp for large $t$ if $D$ is bounded.
\end{remark}

We end this section with
bounds for the survival probabilities in the complement of the ball.
Noteworthy the constants in the bounds do not depend on the radius.

\begin{thm}\label{exit_ub_R2} Suppose that $\psi\in \WLSC{\la}{0}{\lC}\cap\WUSC{\ua}{0}{\uC}$. Let   $R>0$ and $D=\overline{B}^c_R$.
\begin{description}
\item[(i)]  There is a constant $C^*=C^*(d,\la,\,\lC,\,\ua,\,\uC) $ such that,
	$$\p^x(\tau_{D}>t)\le C^*\left( \frac{V(\delta_D(x))}{\sqrt{t}\wedge V(R)}\wedge 1\right),\qquad t>0. $$
\item[(ii)]  If $d> \ua$, then
		$$\p^x(\tau_{D}>t)\approx  \frac{V(\delta_D(x))}{\sqrt{t}\wedge V(R)}\wedge1,
\qquad t>0, $$
	where the comparability constant  depends only on $d,\,\la,\,\lC,\,\ua,\,\uC$. 
\end{description}
 \end{thm}

\begin{proof} In the proof we make the convention that all starred constants may only depend on $d,\,\la,\,\lC,\,\ua,\,\uC$. By Remark \ref{1_dim_ball} we only need to deal with the first part only for $d\ge 2$.
 By the assumption on  $\psi$ and  Proposition \ref{inuV}, $\inf_{R>0}\mathcal{J}(R)\ge c^*_1>0$. Furthermore,  for $d\geq 2$, by Lemma \ref{lem:MHR} or Lemma 
\ref{rem:MHR} we have $\H_\infty<\infty$.  
The first claim now follows from Lemma \ref{exit_ub_R}.

Let $d>\ua$. By (\ref{halfline}) we have
absolute constant $c_2$ such that
$$\p^x(\tau_{D}>t)\ge c_2  \frac{V(\delta_D(x))}{\sqrt{t}}\wedge1,\qquad
t>0,\quad x\in \Rd. $$
Therefore, it is enough to show that there is $c^*_3$ such that
\begin{equation}\label{infty2}\p^x(\tau_{D}=\infty )\ge c^*_3 \left( \frac{V(\delta_D(x))}{V(R)}\wedge1\right).  \end{equation}
Since $d>\ua$, by Corollary \ref{hit_infty}, there is $c^*_4\ge 2$ such that for $|x|\ge c^*_4R$, $\p^x(\tau_{D}=\infty )\ge 1/2$.   It is now enough to show (\ref{infty2}) for $R\le |x|\le 3R/2$. Let $F= B(\frac{3Rx}{2|x|},\frac R2)$.   By the strong Markov property,
$$\p^x(\tau_{D}=\infty )\ge  \E^x \left( \p^{X_{\tau_F}} (\tau_{D}=\infty), |X_{\tau_F}|\ge c^*_4R \right)\ge (1/2) \p^x  ( |X_{\tau_F}|\ge c^*_4R ).$$
By the Ikeda-Watanabe formula,
$$\p^x  ( |X_{\tau_F}|\ge c^*_4R )\ge \nu(\{y: y_1\ge c^*_4 R \})\,\E^x\tau_F .$$
By Theorem \ref{Exit2} and subaddativity of $V$ we have $\E^x\tau_F\ge c^*_5 V(\delta_D(x))V(R)$.   By Lemma \ref{GApprox} and Proposition  \ref{ch1Vp} for $X^1_t$ and subaddativity of $V$ we obtain $\nu(\{y: y_1\ge c^*_4 R \})\ge c^*_6/V^2(R)$  for some $ c^*_6>0$. This proves
(\ref{infty2}).
\end{proof}

We note that the assumption $d> \ua$ cannot in general be removed from the second part of the theorem. For example, if $d=1$, then the
survival probability of the Cauchy process has asymptotics of logarithmic type, see \cite[Remark~10]{MR2722789}.

\begin{remark} Lower-bound counterparts of Lemma \ref{exit_ub_R} and Theorem \ref{exit_ub_R2}(i) follow from Proposition \ref{lalfline}.
\end{remark}

Precise estimates of the tails of the hitting time of the ball for the isotropic stable L\'evy processes are given in \cite{MR2722789}.
For the Brownian motion, \cite{MR3034599} gives even more-precise estimates of the derivative of the survival probability.

\begin{remark}
We conclude this section with an obvious but necessary remark: if $B,\overline{B'}\subset \Rd$ are balls (open and closed, correspondingly) and $B\subset D\subset \overline{B'}^c$, then the survival probability of $D$ is bounded as follows
$$
\p^x(\tau_{B}>t)\le \p^x(\tau_{D}>t)\le \p^x(\tau_{\overline{B'}^c}>t), \qquad x\in \Rd,\, t\ge 0.
$$
This leads to immediate bounds  for the survival probabilities for general $C^{1,1}$ open sets  $D\subset \Rd$: If $\psi\in \WLSC{\la}{0}{\lC}\cap\WUSC{\ua}{0}{\uC}$ and $D$ is $C^{1,1}$ at scale $r$, then by Proposition \ref{L5a}, Remark \ref{rem:spball_exp} and Theorem \ref{exit_ub_R2}, there is $C^*=C^*(d,\la,\ua,\lC,\uC)$ 
such that  
if $x\in \Rd$ and $t\le C_{11}V^2(r)$, then
\begin{equation}
\frac1{C^*}\left( \frac{V(\delta_D(x))}{\sqrt{t}}\wedge 1\right)\leq \p^x(\tau_{D}>t)\le C^*\left( \frac{V(\delta_D(x))}{\sqrt{t}}\wedge 1\right).
\end{equation}
\end{remark}

\section{Discussion of assumptions and applications}\label{sec:conditionA}
\subsection{Condition $\A $}\label{ss:cA}

Recall that function $v>0$ is called log-concave if $\log v$ is concave, and if this is the case,
then the (right hand side) derivative $v'$ of $v$ exists
and $v'/v$ is non-increasing.
The next lemma shows that $\As$ is satisfied with $H_\infty=5$ if $V$ is log-concave.
\begin{lem}\label{delta_V} If $V$ is log-concave and ${0<}x{\le y\le z\le}5x$,  then
$V({z})-V({y})\le 5 V^\prime(x){(z-y)}$.
\end{lem}
\begin{proof}
We have $V>0$ increasing, and $V'/V>0$ non-increasing.
Therefore,
$$\log V(z)-\log V(y)=\int_y^z \frac{V'(s)}{V(s)}ds\le \frac {V^\prime(x)}{V(x)}(z-y),
$$
and
$$ \log V(z)-\log V(y)=\int_{V(y)}^{V(z)}\frac1u\, du\ge \frac{V(z)- V(y)}{V(z)}.$$
By this and
subadditivity of $V$,
$$ V(z)- V(y)\le \frac { V(z)V^\prime(x)}{V(x)}(z-y)\le  5 V^\prime(x)(z-y).$$
\end{proof}

The next lemma shows that for dimension $d\ge 3$, the weak lower scaling condition implies
$\A $, while the week global lower scaling implies
$\As$. This helps extend many results previously known only  for complete subordinate Brownian motions
with
scaling (see below for definitions).

\begin{remark}\label{rmV}
A sufficient condition for log-concavity of $V$ is that $V'$ be monotone, which is common for subordinate Brownian motions, for instance if the subordinator is special. For complete subordinate Brownian motions, $V$ is even a Bernstein function   
(see \cite[Proposition 4.5]{KMR}).
It is interesting to note that $V'$ is not monotone for the so-called truncated $\alpha$-stable L\'evy processes with $0<\alpha<2$ \cite{MR2282263}.
Indeed, if the L\'evy measure has compact support, then by \cite[(5.3.4)]{MR2320889}  the L\'{e}vy measure of the ladder-height process (subordinator) has compact support as well.   
By \cite[Proposition~10.16]{MR2978140}, $\kappa$ is not a special Bernstein function, therefore by \cite[Theorem 10.3]{MR2978140}, $V'$ is not decreasing. We, however, note that the truncated stable processes have global weak lower scaling with $\la=\alpha$,
and our estimates of the expected exit time for the ball hold for these processes with the comparability constant independent of $r$. This shows flexibility of our methods.
\end{remark}

\begin{lem}\label{lem:MHR} If $d\ge 3$ and $\psi\in$WLSC$(\beta, \theta, \lC)$, then $\A $ holds with  $H_R=H_R(\beta, \theta , \lC, R)$  for any $R\in(0,\infty)$. If, furthermore, $\theta=0$, then $\As$ even holds.
\end{lem}
\begin{proof}
By \cite[Corollary 5]{2013arXiv1301.2441G},
the scale invariant  Harnack inequality holds for $X^1$, the one-dimensional projection of $X$. Namely, for every
$R>0$ there is $C_R<\infty$ such that if $0<r\leq R$, $h\ge 0$ on $\R$   and
$h$ is harmonic for $X^1$ on $(-r,r)$, then
$$\sup_{y\in (-r/2,r/2)}h(y)\leq C_R\inf_{y\in  (-r/2,r/2)}h(y).$$
Since $V^\prime$ is harmonic on $(0, \infty)$  for $X^1$ and $x_0\geq 2r$, then by spatial homogeneity,
$$\sup_{\theta\in(x_0-r,x_0+r)}V^\prime(\theta)\leq C_R \inf_{\theta\in(x_0-r,x_0+r)}V^\prime(\theta).$$
Using the inequality with $(x_0,r)=(x,x/2),\,(9x/4,x),\,(4x,x)$, where $0<x\leq R$ we get
$$\sup_{\theta\in(x/2,5x)}V^\prime(\theta)\leq C^3_R \inf_{\theta\in(x/2,5x)}V^\prime(\theta)\le C^3_R V^\prime(x).$$
The absolute continuity of $V$ yields the conclusion.
\end{proof}

\begin{lem}\label{rem:MHR}
Let $d\ge 1$ and $\psi\in\WLSC{\la}{ \theta}{ \lC}\cap\WUSC{\ua}{ \theta}{ \uC}$. Then
$\A $ holds with  $H_R=H_R(\la,\ua, \theta , \lC, \uC,R)$ for all $R\in(0,\infty)$. If, furthermore, $\theta=0$, then $\As $ even holds.
\end{lem}
\begin{proof} By the same arguments as given in Lemma \ref{lem:MHR} it is enough to show that the scale invariant Harnack inequality holds for $X^1$.
{
By \cite[Corollary 22 and (16)]{2013arXiv1305.0976B}  and Proposition \ref{ch1Vp} applied to $X^1$, there exists $r_0>0$ such that
$$\nu_1(u)\approx \frac{1}{V(|u|)^2|u|^{d}}, \quad 0<|u|<r_0/\theta.$$
At first, let $\theta>0$. By \cite[Theorem 5.2]{MR2524930} used with auxiliary function $$
\phi(r)=
\left\{
\begin{array}{ll}
V^2(r) & \text{ if }0<r\leq r_0/\theta,\\
V^2(r_0/\theta)(r\theta/r_0)^{\ua}& \text{ if } r>r_0/\theta,
\end{array}
\right.
$$
 we infer that the scale invariant Harnack inequality holds for $X^1$.  

For $\theta=0$ we use \cite[Theorem 4.12]{MR2357678} instead of \cite[Theorem 5.2]{MR2524930} to get the global scale invariant Harnack inequality for $X^1$. In consequence we obtain
$\As $.
}
\end{proof}

{\begin{lem}\label{CondAsigma1}
If $\sigma>0$, then
$\A $ holds.
\end{lem}}
\begin{proof}
If $\sigma>0$, then  $V'$ is positive, continuous  and bounded by $\sigma^{-1}$ (see Lemma \ref{kappaDrift}).
By Cauchy's mean value theorem, for $R>0$ we have
$$V(z)-V(y)\leq \sigma^{-1}(z-y)\leq H_RV'(x)(z-y)\quad 0<x\leq y\leq z\leq 5R,$$
where $H_R=(\sigma\inf_{z\leq 5R}V'(z))^{-1}<\infty$.
\end{proof}

The case when $X$ is a subordinate Brownian motion is of special interest in this theory:
we consider a Brownian $B$ motion in $\Rd$ and an independent subordinator $\eta$,
and we let
$$
X(t)=B(2\eta(t)).
$$
The process $X$ is then called a subordinate Brownian motion.
The monograph \cite{MR2978140} is devoted to the study of such processes.
Furthermore, $X$ is called a {\it special} subordinate Brownian motion if the subordinator is {\it special} (i.e. given by a special Bernstein function \cite[Definition~10.1]{MR2978140}), and it is called {\it complete} subordinate Brownian motion if the subordinator is even {\it complete} \cite[Proposition 7.1]{MR2978140}.
We let $\varphi$
be the Laplace exponent of the subordinator, i.e.
$$
\E \exp[-u\eta(t)]=\exp [-t\varphi(u)], \quad u\ge 0.
$$
Since
$$\E e^{i<\xi,B_t>}=e^{-t|\xi|^2/2}, \quad t\ge 0, \; \xi\in \Rd,$$
we have
$$\psi(\xi)=\varphi(|\xi|^2).$$
Then by \cite[{Theorem~4.4}]{KMR},
$V(r)\approx \varphi(r^{-2})^{-1/2}$.
For clarity, \cite{KMR} makes the assumption that $\varphi$ is unbounded, but it is not necessary for the result.
In connection to \cite[Remark 4.7]{KMR} we note
that $\varphi(x)$ and $x/\varphi(x)$ are monotone. For instance,
by concavity, if $s\ge 1$ and $x\ge 0$, then $\varphi(sx)  \le s \varphi(x)$, hence $sx/\varphi(sx) \ge x/\varphi(x)$.

\begin{remark}
If $X$ is a subordinate Brownian motion, then due to \cite[Theorem 7]{2013arXiv1301.2441G} we may
skip the assumption $d\ge 3$ in Lemma~\ref{lem:MHR}. This is related to the fact that Harnack inequality is inherited by orthogonal projections of isotropic unimodal L\'evy processes, and  every subordinate Brownian motion in dimensions $1$ and $2$ is a projection of a subordinate Brownian motion in dimension $3$. (This observation was used before in \cite{MR2928720}.)
\end{remark}

\begin{lem}\label{specialConcave}
If $X$ is a special subordinate Brownian motion,
then $V$ is concave.
\end{lem}
\begin{proof}
By \cite[Proposition 2.1]{MR2513121}, the Laplace exponent $\kappa$ given by \eqref{kappa} is a special Bernstein function. In fact, \cite{MR2513121} makes the assumption that the Laplace exponent of a subordinator  is a complete Bernstein function, but the same proof  works
if it is only a special Bernstein function,  since
it suffices that $|x|^2/\psi(x)$ be negative definite.
 Then \cite[Theorem 10.3]{MR2978140} implies that $V^\prime$ is non-increasing, which ends the proof.
\end{proof}

\begin{remark}\label{A*ssBm}
Lemma~\ref{specialConcave} implies
$\As $ with $H_\infty=1$ for special subordinate Brownian motions.
\end{remark}

We finish this section with a simple argument leading to  boundary Harnack inequality.
\begin{prop}\label{prop:BHP}Let $\nu$ be continuous in $\Rdz$.
Assume that $\psi$ satisfies the global weak lower and upper scaling conditions, $D$ is $C^{1,1}$  at scale $\rho>0$,
$z\in \partial D$, $0<r<\rho$ and $u \geq 0$ is regular harmonic in $D\cap B(z,r)$ and vanishes in $B(z,r)\setminus D$.
Then positive $c=c(d,\psi)$, $c_1=c_1(d,\psi)$ exist such that
$$\frac{u(x)}{u(y)}\leq c\frac{V(\delta_{D}(x))}{V(\delta_{D}(y))}
\le c_1\sqrt{\frac{\psi(1/\delta_{D}(y))}{\psi(1/\delta_{D}(y))}}
,\qquad x,y\in D\cap B(z,r/2).$$
\end{prop}
\begin{proof}
The first inequality is a consequence of \cite[Lemma 5.5]{2012ScChA..55.2317K}. Namely \cite[Corollary 27]{2013arXiv1305.0976B} shows that the assumptions of \cite{2012ScChA..55.2317K} are satisfied. Then we estimate the expected exit time of $D\cap B(z,r)$ by using
Lemma~\ref{rem:MHR}, Theorem~\ref{eq:eet}, Corollary~\ref{exit_time_R} and Proposition~\ref{inuV}. The second inequality follows from Proposition~\ref{ch1Vp} and \eqref{eqcpg}. 
\end{proof}

\subsection{Examples}\label{sec:ex}

Our results apply to the following unimodal L\'evy processes.
In each case 
our sharp bounds for the expected first exit time from the ball apply
and the comparability constants depend only on the dimension and the L\'evy-Kchintchine exponent of the process but not on the radius of the ball.
Our estimates of the probability of surviving in $B_r$ and $\overline{B_r}^c$ also hold with constants independent of $r$ if the characteristic exponent of $X$ has global upper and lower scalings (see \cite{2013arXiv1305.0976B}, for a simple discussion of scaling). If the scalings are not global, then the constants may deteriorate as $r$ increases.
\begin{exmp} Chapter~15 of \cite{MR2978140} lists
more than one hundred cases and classes of complete  Bernstein functions.
All of those which are unbounded and have killing rate $0$ are covered by our results (see Lemma \ref{specialConcave}): we obtain sharp estimates of the  expected first exit time from the ball. In fact, the comparability constants depend only on the dimension. This is, e.g., the case for L\'evy process with the characteristic exponent
$$\psi(\xi)=\[|\xi|^{\alpha_2}+(|\xi|^2+m)^{\alpha_3/2}-m^{\alpha_3/2}\]^{1-\alpha_1/2}\log^{\alpha_1/2}(1+|\xi|^{\alpha_4}),$$ where
$\alpha_1,\alpha_2,\alpha_3,\alpha_4\in [0,2]$, $\alpha_1+\alpha_2+\alpha_3>0$, $\alpha_2+\alpha_3+\alpha_4>0$  and $m\geq 0$, and also when
$$\psi_2(\xi)=u(|\xi|)-u(0^+),$$
 where $u(r)=mr^2+r^2/\log^{\alpha_1/2}(1+r^{\alpha_4})$.
 These include, e.g., isotropic stable process, relativistic stable process, sums of two independent  
isotropic stable processes (also with Gaussian component) and geometric stable processes, variance gamma processes and conjugate
to geometric stable processes \cite{MR2978140}.
\end{exmp}

\begin{exmp}
Let $0<\alpha_0\le \alpha_1\le \ldots\le 2$, $\alpha^*=\lim_{k\to\infty}a_k$, and
define $f(r)=r^{-\alpha_{[r]}}$, $r>0$.
Then $f(1/r)\in \WLSC{\alpha_0}{0}{1}\cap \WUSC{\alpha^*}{0}{1}$.
Consider a unimodal L\'{e}vy process with L\'{e}vy density $\nu(x)=f(|x|)|x|^{-d}$, $x\neq 0$. By \cite[Proposition 28]{2013arXiv1305.0976B} and \cite[Proposition 8]{2013arXiv1301.2441G},  
$\psi\in \WLSC{\alpha_0}{0}{\lC}\cap \WUSC{\alpha^*}{0}{\uC}$. For $d\geq 3$ by Lemma \ref{lem:MHR} we get $\E^x\tau_{B_r}\approx 1/\sqrt{f(r)f(r-|x|)}$, where $|x|<r$, $r>0$, and the comparability constant is independent of $r$. If $\alpha^*<2$, then the comparability holds for $d=2$, too, cf. Lemma \ref{rem:MHR}.
\end{exmp}

\begin{exmp} 
Let $d\geq 3$, $\sigma\geq 0$, $\nu(x)=f(|x|)/|x|^d$, $x\in\Rdz$. Let $f\ge 0$ be non-increasing and let $\beta>0$ be such that $f(\lambda r)\leq c\lambda^{-\beta}f(r)$ for $r>0$ and $\lambda>1$ (see \cite[Example 2 and 48]{2013arXiv1301.2441G} and Lemma \ref{lem:MHR}).
So is the case for the following processes (with $\alpha,\alpha_1\in(0,2)$):
truncated stable process ($f(r)=r^{-\alpha}\textbf{1}_{(0,1)}(r)$),
tempered stable process ($f(r)=r^{-\alpha}e^{-r}$),
isotropic Lamperti stable process ($f(r)=re^{\delta r}(e^s-1)^{-\alpha-1}$, where $\delta<\alpha+1$) and
layered stable process ($f(r)=r^{-\alpha}\textbf{1}_{(0,1)}(r)+ r^{-\alpha_1}\textbf{1}_{[1,\infty)}(r)$).
\end{exmp}

More examples of isotropic processes with scaling may be found in Section~\cite[Section~4.1]{2013arXiv1305.0976B}.

		\end{document}